\documentclass[10pt]{article}

\usepackage{amsmath,amsfonts,amsthm,amscd,amssymb,graphicx,mathrsfs}
\numberwithin{equation}{section}

\usepackage{fullpage}

\usepackage{color}

\newtheorem{theorem}{Theorem}[section]

\newtheorem{lemma}[theorem]{Lemma}

\newtheorem{proposition}[theorem]{Proposition}

\newtheorem{remark}[theorem]{Remark}

\begin{document}

\title{Ill-posedness of the hydrostatic Euler and singular Vlasov equations}

%(some!) kinetic models}

\author{Daniel Han-Kwan\footnotemark[1] \and Toan T. Nguyen\footnotemark[2]}

\date{\today}

\maketitle

\renewcommand{\thefootnote}{\fnsymbol{footnote}}

\footnotetext[1]{CMLS, \'Ecole polytechnique, CNRS, Universit\'e Paris-Saclay, 91128 Palaiseau Cedex, France. \\Email: daniel.han-kwan@polytechnique.edu}
\footnotetext[2]{Department of Mathematics, Pennsylvania State University, State College, PA 16802, USA. \\Email: nguyen@math.psu.edu}

\begin{center}
\emph{This paper is dedicated to Claude Bardos 
\\on the occasion of his 75th birthday, as a token of friendship and admiration.}
\end{center}
\bigskip

\begin{abstract}

 In this paper, we develop an abstract framework to establish ill-posedness in the sense of Hadamard for some nonlocal PDEs displaying unbounded unstable spectra. We apply it to prove the ill-posedness for the hydrostatic Euler equations as well as for the kinetic incompressible Euler equations and the Vlasov-Dirac-Benney system.

\end{abstract}

%\tableofcontents 

\section{Introduction}

 In this paper, we develop an abstract framework to establish ill-posedness in the sense of Hadamard for some nonlocal PDEs displaying unbounded unstable spectra;
 this phenomenon is reminiscent of Lax-Mizohata ill-posedness for first-order systems violating the hyperbolicity condition (that is, when the spectrum of the operator's principal symbol is not included in the real line).

 \bigskip
 
By (local-in-time) well-posedness of the Cauchy problem for a PDE, we mean

\begin{itemize}

\item given initial data, there exists a time $T>0$ so that a solution exists for all times $t \in [0,T]$; 

\item the solution is unique; 

\item the solution map is (H\"older) continuous with respect to initial data. 

\end{itemize}
This notion of well-posedness for PDEs was introduced by Hadamard \cite{Ha}. %, in which his third condition for well-posedness requires only continuous dependence of the solution with respect to initial data.  
The Lax-Mizahota ill-posedness result is concerned with this definition. \emph{In this paper, we  shall describe situations in which the third well-posedness condition breaks down (for data in Sobolev spaces).}  

\bigskip

In the context of systems of first-order partial differential equations, Hadamard's well-posedness was extensively studied by many authors, including Friedrichs, G\r{a}rding, H\"ormander, Lax, among others; see, for instance, \cite{Ha, Lax, Mi, Wa, Lebeau, BCT, M, LNT} and the references therein. For linear equations, it was Lax \cite{Lax} and Mizohata \cite{Mi} who first showed that hyperbolicity is a necessary condition for well-posedness of the Cauchy problem for $C^\infty$ initial data. The result was later extended to quasilinear systems by Wakabayashi \cite{Wa}, and recently by M\'etivier \cite{M}. The violation of hyperbolicity creates unbounded unstable spectrum of the underlying principal differential operators. M\'etivier \cite{M} showed that for first order systems that are not hyperbolic, the solution map is not $\alpha$-H\"older continuous (for all $\alpha \in (0,1]$) from any Sobolev space to $L^2$, or more precisely, it does not belong to $C^\alpha({H}^s, L^2)$, for all $s \geq 0$ and $\alpha \in (0,1]$, within arbitrarily short time; that is, the Cauchy problem is ill-posed, violating the above third condition for well-posedness.    

In this paper, we prove an analogue of M\'etivier's result for some nonlocal PDEs: namely, the hydrostatic Euler equations as well as for some singular Vlasov equations: the kinetic incompressible Euler equations and the Vlasov-Dirac-Benney system. The purpose of this introduction is to briefly discuss these equations and to present the main results. Our abstract framework for proving the ill-posedness is inspired by the analysis of M\'etivier \cite{M} and Desjardins-Grenier \cite{DG}.

%\cite{BCT}

\subsection{Hydrostatic Euler equations}

The Hydrostatic Euler equations arise in the context of two-dimensional incompressible ideal flows in a narrow channel (see e.g. \cite{Lions}). They read:
\begin{equation}\label{hydro}
\left \{ 
\begin{aligned}
\partial_t  u +  u \partial_x  u +  v \partial_z  u +  p_x &=0,
\\
\partial_x  u + \partial_z  v &=0,
\end{aligned}
\right.
\end{equation}
for $(x,z) \in \mathbb{T} \times [-1,1]$, where $\mathbb{T} := \mathbb{R}/\mathbb{Z}$. The torus $\mathbb{T}$ is equipped with the normalized Lebesgue measure, so that $\mathrm{Leb} ( \mathbb{T} )= 1$. Here $( u(t,x,z), v(t,x,z))\in \mathbb{R}^2$, and $ p(t,x)$ are the unknowns in the equation. We impose the zero boundary conditions: 
$$  v_{\vert_{z = \pm 1}} = 0.$$
The vorticity $ \omega := \partial_z  u $ satisfies the following equation
\begin{equation}\label{vorti}
\partial_t  \omega +  u \partial_x  \omega +  v \partial_z  \omega = 0,
\end{equation}
in which $ u := \partial_z  \varphi$, $ v := -\partial_x  \varphi$, and the stream function $ \varphi$ solves the elliptic problem: 
\begin{equation}\label{def-stream-in}
\partial_z^2  \varphi =  \omega , \qquad  \varphi_{\vert_{z=\pm 1}} = 0.
\end{equation}
Thus, one can observe a loss of one $x$-derivative in the equation \eqref{vorti} through $ v =- \partial_x  \varphi$, as compared to $ \omega$. This indicates that a standard Cauchy theory cannot be expected for this equation.

\bigskip

Brenier was the first to develop a Cauchy theory in Sobolev spaces for data with convex profiles \cite{Br99}; this was revisited  and extended recently in Masmoudi-Wong \cite{MW} and Kukavica-Masmoudi-Vicol-Wong \cite{KMVW}. In \cite{KTVZ}, Kukavica-Temam-Vicol-Ziane also provide an existence result for data with analytic regularity. 

The derivation of Hydrostatic Euler from the incompressible Euler equations set in a narrow channel, for data with convex profiles, was first performed by Grenier \cite{Gr99-2}, then by Brenier \cite{Br03} with different methods (see also \cite{MW}). One key idea in these works is the use of the convexity  to build a suitable energy which is not degenerate in the hydrostatic limit.

In \cite{Renardy}, Renardy showed that for arbitrary odd shear flows $U(z)$ so that $\frac{1}{U(z)^2}$ is integrable, the linearized hydrostatic Euler equations \eqref{vorti}-\eqref{def-stream-in} around $U'$ have unbounded unstable spectrum. Such profiles do not satisfy the convexity condition. Following an argument of \cite{GT}, this property for the spectrum can be used to straightforwardly prove some ill-posedness for the nonlinear equations (see also \cite{GVD, GVN, GN} for the ill-posedness of the Prandtl equations or \cite{FV} for the SQG equations); loosely speaking, it asserts that the flow of solutions, if it exists, cannot be $C^1(H^s, H^1)$, for all $s \geq 0$, within a fixed positive time. In this work, we shall construct a family of solutions to show that the solution map from $H^s$ to $L^2$ has unbounded H\"older norm, within arbitrarily short time. In the proof, we shall take an unstable shear flow that is analytic. Such a shear flow exists; for instance, $U(z) = \mathrm{tanh}(\frac z{d_1})$ for small $d_1$ yields unstable spectrum as shown by \cite{CM}.

%What we shall prove in this paper is that the flow cannot be $C^\alpha({H}^s_{\mathrm{weight}}, L^2)$, for all $s \geq 0$, $\alpha \in (0,1]$, and any polynomial weight in $v$.
%

\bigskip

We prove the following ill-posedness result: 
 
\begin{theorem}[Ill-posedness for the hydrostatic Euler equations]
\label{t-hydro}
There exists a stationary shear flow $U(z)$ such that the following holds.
For all $s \in \mathbb{N}$, $\alpha \in (0,1]$, and $k \in \mathbb{N}$, there are families of solutions $(\omega_\varepsilon)_{\varepsilon>0}$ of \eqref{vorti}-\eqref{def-stream-in}, times $ t_\varepsilon = \mathcal{O}(\varepsilon |\log \varepsilon|)$, and $(x_0, z_0) \in \mathbb{T} \times (-1,1)$ such that
\begin{equation}
\lim_{\varepsilon \to 0} \frac{ \| \omega_\varepsilon - U' \|_{L^2([0,t_\varepsilon] \times \Omega_\varepsilon)}}{\| \omega_\varepsilon|_{t=0} -U' \|^\alpha_{H^s(\mathbb{T} \times (-1,1))}}=+\infty
\end{equation}
with $\Omega_\varepsilon =B(x_0, \varepsilon^k) \times B(z_0, \varepsilon^k)$.%, $x_0 \in \mathbb{T}^3$, and $v_0 \in \mathbb{R}^3$. 

\end{theorem}

We remark that the instability is strong enough so that it occurs within a vanishing spatial domain $\Omega_\varepsilon$ and a vanishing time $t_\varepsilon$, as $\varepsilon \to0$. As will be seen in the proof, $(x_0, z_0)$ can actually be taken arbitrarily in $\mathbb{T} \times (-1,1)$. 

\subsection{Kinetic incompressible Euler and Vlasov-Dirac-Benney equations}

%In this paper, we prove ill-posedness in the sense of Hadamard of 
The so-called {kinetic incompressible Euler} and Vlasov-Dirac-Benney systems are kinetic models from plasma physics, arising in the context of \emph{small Debye lengths regimes}. Although our results will be stated in the three-dimensional framework, they can be adapted to any dimension.

Consider first the kinetic incompressible Euler equations, which read
\begin{equation}\label{Vlasov}
\partial_t f + v\cdot \nabla_x f  - \nabla_x \varphi \cdot \nabla_v f = 0, 
\end{equation}
\begin{equation}
\label{Euler}
\rho(t,x) := \int_{\mathbb{R}^3} f(t,x,v) \, dv = 1, 
\end{equation}
for $(t,x,v) \in \mathbb{R}^+ \times \mathbb{T}^3 \times \mathbb{R}^3$, in which $f(t,x,v)$ is the distribution function at time $t\ge 0$, position $x \in \mathbb{T}^3 := \mathbb{R}^3 /\mathbb{Z}^3$, and velocity $v \in \mathbb{R}^3$ of electrons in a plasma. The torus $\mathbb{T}^3$ is equipped with the normalized Lebesgue measure, so that $\mathrm{Leb} ( \mathbb{T}^3 )= 1$. The potential $\varphi$ stands for a Lagrange multiplier (or, from the physical point of view, a {pressure}) related to the constraint $\rho=1$. 
It is possible to obtain an explicit formula for the potential $\varphi$, arguing as follows. 
Introduce the current density  $j (t,x) := \int_{\mathbb{R}^3} f(t,x,v) v \, dv$.
We start by writing the local conservation of charge and current from the Vlasov equation:
$$
\begin{aligned}\partial_t \rho + \nabla \cdot j&= 0,
\\
\partial_t j + \nabla \cdot \int f v \otimes v \, dv &= -\nabla \varphi.
\end{aligned}$$
By using the constraint \eqref{Euler}, it follows that $\nabla \cdot j=0$. Plugging this into the conservation of current, one gets the law
 \begin{equation}
\label{incomp}
-\Delta \varphi =  \nabla \cdot \left(\nabla \cdot \int f v \otimes v \, dv  \right).
\end{equation}
Looking for solutions to \eqref{Vlasov}-\eqref{Euler} of the form $f(t,x,v) = \rho(t,x) \delta_{v=u(t,x)}$
turns out to be equivalent to finding solutions $(\rho,u)$ of the classical incompressible Euler equations. This therefore justifies the name we have chosen for \eqref{Vlasov}-\eqref{Euler}, as suggested by Brenier \cite{Br89}.

\bigskip

The {Vlasov-Dirac-Benney system is closely related. It reads
  \begin{equation}
\label{comp1}
\partial_t f + v\cdot \nabla_x f  - \nabla_x \varphi \cdot \nabla_v f = 0, 
\end{equation}
  \begin{equation}
\label{comp2}
 \varphi =  \int_{\mathbb{R}^3} f(t,x,v) \, dv  -1.
\end{equation}
This model appears to be a kinetic analogue of the compressible isentropic Euler equations with parameter $\gamma=2$. %\emph{via} monokinetic distributions $f(t,x,v) = \rho(t,x) \delta_{v=u(t,x)}$. 
The name Vlasov-Dirac-Benney was coined by Bardos in \cite{Bardos}, due to connections with the Benney model for Water Waves.

\bigskip

Both kinetic incompressible Euler and Vlasov-Dirac-Benney equations can be formally derived in the \emph{quasineutral limit} of the Vlasov-Poisson system, i.e. in the small Debye length regime. This corresponds to the singular limit $\varepsilon \to 0$ in the following scaled equations:
$$
\partial_t f_\varepsilon + v\cdot \nabla_x f_\varepsilon  - \nabla_x \varphi_\varepsilon \cdot \nabla_v f_\varepsilon = 0, 
$$
where $\varphi_\varepsilon$ solves a Poisson equation given,
\begin{enumerate}
\item {\bf for the case of electron dynamics}, by
$$
-\varepsilon^2 \Delta_x \varphi_\varepsilon = \rho_\varepsilon -1, \qquad \rho_\varepsilon := \int_{\mathbb{R}^3} f_\varepsilon \, dv,
$$
which yields the kinetic incompressible Euler equations in the formal limit $\varepsilon \to 0$ (see \cite{Br89});
\item {\bf for the case of ion dynamics}, by
$$
-\varepsilon^2 \Delta_x \varphi_\varepsilon = \rho_\varepsilon - \varphi_\varepsilon - 1, \qquad \rho_\varepsilon := \int_{\mathbb{R}^3} f_\varepsilon \, dv,
$$
which yields the Vlasov-Dirac-Benney system in the formal limit $\varepsilon \to 0$ (see \cite{HK11}).
\end{enumerate}

\bigskip

Directly from the laws \eqref{incomp} and \eqref{comp2} for the potential $\varphi$, one sees that there is a loss of one $x$-derivative through the force $-\nabla_x \varphi$, as compared to the distribution function $f$. This explains why a standard Cauchy theory cannot be expected for these equations. What is known though is the existence of \emph{analytic solutions} (see \cite{HK11}, Jabin-Nouri \cite{JN}, Bossy-Fontbona-Jabin-Jabir \cite{BFJJ}), as well as an $H^s$ theory for \emph{stable} data (see  Bardos-Besse \cite{BB} and the recent work of the first author and Rousset \cite{HKR}).

As for the rigorous justification of the quasineutral limit, we first refer to the work of Grenier \cite{Gr95} in the case of data with analytic regularity in $x$ (see also \cite{HI,HI2} where it is shown that exponentially small but rough perturbations of the data considered by Grenier are admissible). In  \cite{B}, Brenier introduced the so-called modulated energy method and derived the incompressible Euler equations in the limiting case of monokinetic distributions (see \cite{HK11} for what concerns the case of the compressible isentropic Euler system). In the work \cite{HKH}, the first author and Hauray showed that the formal limit (to \eqref{Vlasov}-\eqref{Euler} or \eqref{comp1}-\eqref{comp2}) is in general not true in Sobolev spaces, because of instabilities of the Vlasov-Poisson system (see also \cite{HKN1}).
The rigorous derivation of \eqref{comp1}-\eqref{comp2} for initial data with a Penrose stability condition was completed only recently by the first author and Rousset \cite{HKR}.

In \cite{BN}, Bardos and Nouri show that around \emph{unstable} homogeneous equilibria, the linearized equations of \eqref{comp1}-\eqref{comp2} have unbounded unstable spectrum. This property was used to prove some ill-posedness, using the above-mentioned argument of \cite{GT}, see \cite[Theorem 4.1]{BN}; loosely speaking they show that the flow of solutions, if it exists, cannot be $C^1(H^s, H^1)$, for all $s \geq 0$. What we shall prove in this paper is that the flow cannot be $C^\alpha({H}^s_{\mathrm{weight}}, L^2)$, for all $s \geq 0$, $\alpha \in (0,1]$, and any polynomial weight in $v$. In the proof we shall take unstable homogeneous equilibria that are analytic and decaying sufficiently fast at infinity: typical examples are double-bump equilibria satisfying these constraints.

\bigskip
More precisely, we prove the following ill-posedness result: 

\begin{theorem}[Ill-posedness for the kinetic incompressible Euler and Vlasov-Dirac-Benney equations]
\label{t-ill}
There exists a stationary solution $\mu(v)$ such that the following holds.
For all $m,s \in \mathbb{N}$, $\alpha \in (0,1]$, and $k \in \mathbb{N}$, there are families of solutions $(f_\varepsilon)_{\varepsilon>0}$ of \eqref{Vlasov}-\eqref{Euler} (respectively, of the system \eqref{comp1}-\eqref{comp2}), times $ t_\varepsilon = \mathcal{O}(\varepsilon |\log \varepsilon|)$, and  $(x_0, v_0) \in \mathbb{T}^3 \times \mathbb{R}^3$, such that
\begin{equation}
\lim_{\varepsilon \to 0} \frac{ \| f_\varepsilon -\mu \|_{L^2([0,t_\varepsilon] \times \Omega_\varepsilon)}}{\| \langle v\rangle^m (f_\varepsilon|_{t=0} - \mu) \|^\alpha_{H^s(\mathbb{T}^3 \times \mathbb{R}^3)}}=+\infty
\end{equation}
with $\Omega_\varepsilon =B(x_0, \varepsilon^k) \times B(v_0, \varepsilon^k)$. Here, $\langle v\rangle := \sqrt{1+|v|^2}$. 
%, $x_0 \in \mathbb{T}^3$, and $v_0 \in \mathbb{R}^3$. 

\end{theorem}

In the proof, we shall focus only on the system \eqref{Vlasov}-\eqref{incomp}, since the analysis is almost identical for what concerns the Vlasov-Dirac-Benney equations. Furthermore, this result also holds for \eqref{Vlasov}-\eqref{Euler} and  \eqref{Vlasov}-\eqref{incomp} in any dimension $d \in \mathbb{N}^*$.

\bigskip

The abstract ill-posedness framework will be presented in Section \ref{s-abstract}.
The ill-posedness of the hydrostatic Euler and kinetic incompressible Euler equations is then proved in Section \ref{s-Euler} and Section \ref{s-ill}, respectively.

\section{An abstract framework for ill-posedness }\label{s-abstract}

In this section, we present a framework to study the ill-posedness of the following abstract PDE for $U = U(t,x,z)$:
\begin{equation}
\label{e-abstract}
\partial_t U - \mathscr{L} U = \mathcal{Q}(U,U), \quad t\geq 0, \, \quad x \in \mathbb{T}^d := \mathbb{R}^d / \mathbb{Z}^d, \, \quad z \in \Omega, %\mathbb{R}^d,
\end{equation}
in which $\mathscr{L}$ (resp. $\mathcal{Q}$) is a linear (resp. bilinear) integro-differential operator in $(x,z)$, $d \in \mathbb{N}^*$ and $\Omega$ is an open subset of $\mathbb{R}^{d'}$, $d' \in \mathbb{N}^*$. 
If $\Omega \neq \mathbb{R}^m$, then some suitable boundary conditions on $\partial \Omega$ are enforced for $U$. The choice of $\mathbb{T}^d$ is made for simplicity, and other settings are possible.

Consider the sequence $\varepsilon_k = \frac{1}{k}$, for $k \in \mathbb{N}^*$. In the following, we forget the subscript $k$ for readability.
Following M\'etivier \cite{M}, we look for solutions $U$ under the form
\begin{equation}
U(t,x,z) \equiv  u\left( \frac{t}{\varepsilon}, \frac{x}{\varepsilon}, z\right),
\end{equation}
%TODO Write first abstract stuff and use hyperbolic scaling.
where $u(s,y,z)$ is $1$-periodic in $y_1,\cdots,y_d$. Assume that one can write
$$
\mathscr{L} U = \left[\frac{1}{\varepsilon} L u + R_1(u)\right] \left(\frac{t}{\varepsilon}, \frac{x}{\varepsilon} , z\right), \quad  \mathcal{Q}(U,U)= \left[\frac{1}{\varepsilon} Q(u,u) + R_2(u,u)\right] \left(\frac{t}{\varepsilon}, \frac{x}{\varepsilon} , z\right),
$$
where $L$ is independent of $\varepsilon$; on the other hand, although we do not write it explicitly, the operators $R_1, Q, R_2$ may depend on $\varepsilon$.
This leads to the study of the following abstract PDE:
\begin{equation}
\label
{e-oscillating}
\partial_s u  - L u = Q(u,u)  + \varepsilon( R_1(u)  +R_2(u,u)), \quad s\geq 0, \quad y \in \mathbb{T}^d,\quad z \in \Omega.
\end{equation}
%The purpose of the hyperbolic scaling is to make use of the {\em ellipticity} underlying in the (leading) linear differential operator $\mathcal{L}$. 
We finally assume the existence of a family of norms $(\| \cdot \|_{\delta, \delta' })_{\delta,\delta'>0}$ satisfying the following properties. For all $\delta, \delta'>0$, the corresponding function space 
$$X_{\delta, \delta'}:= \{ u (y,z), \, \|u \|_{\delta,\delta'} <  + \infty \}$$ 
(with possibly boundary conditions in $z$, if $\Omega \neq \mathbb{R}^{d'}$)
is compactly embedded into~$\langle v\rangle^m$~-~weighted Sobolev spaces $H^s$ (with $\langle z \rangle = \sqrt{1+ |z|^2}$), for all $m,s>0$. 

Moreover, for all  $0\leq \delta < \delta_1$, $0\leq \delta' < \delta_1'$, the following inequalities hold for all $u$:
 $$
\| u \|_{\delta, \delta'} \le  \| u \|_{\delta_1, \delta_1'}, \qquad  \| \partial_y u \|_{\delta, \delta'} \leq  \frac{\delta_1}{\delta_1-\delta}  \|  u \|_{\delta_1, \delta'}, \quad 
  \| \partial_z u \|_{\delta, \delta'} \leq  \frac{\delta_1'}{\delta_1'-\delta'}  \|  u \|_{\delta, \delta_1'}.
 $$

In what follows, we shall carry our analysis on the scaled system \eqref{e-oscillating}. 

\bigskip

We make the following structural assumptions on the abstract PDE \eqref{e-oscillating}.

\begin{enumerate}

%\item[\bf (H.1)] (\emph{Analytic norm}) 

\item[\bf (H.1)] (\emph{Spectral instability for $L$}). 
There exists an eigenfunction  g 
 associated to an eigenvalue $\lambda_0$, with $\Re \lambda_0>0$,  for $L$.
In other words, the set  
$$\Sigma^+ := \Big\{\lambda_0 \in \mathbb{C},  \Re \lambda_0 >0, \, \exists g\neq 0,  \,  L g = \lambda_0 g \, \Big\}$$
 is not empty.

%For each $n \in \mathbb{Z}^d$, let $L_n$ be defined by 
%$$ L_n u: = \langle L (e^{in\cdot y} u), e^{in \cdot y} \rangle_{L^2 (\mathbb{T}^d)}, $$ 
%for $u = u(s,v)$ and assume that TODO: doesn't make sense!
%$$ \gamma_0 : = \max_{\hat n \in \mathbb{S}^{d-1}} \Big\{ \Re \lambda >0 ~:~ \lambda \in \sigma (L_{\hat n})\Big\},$$
%in which $ \sigma (L_{\hat n})$ denotes the spectrum of $L_{\hat n}$. Let $\hat n_0$ and $\lambda_0 \in \sigma (L_{\hat n_0})$ so that $\gamma_0 = \Re\lambda_0$ is positive, and the corresponding eigenfunction $\hat g(v)$ of $L_{\hat n_0}$ satisfies $\| \hat g  \|_{\delta'_0} \lesssim 1,$ for some $\delta'_0>0$; furthermore, there is no open set in $\Omega$ on which $\hat g=0$. In addition, the eigenfunction function $g = e^{i \hat n_0\cdot y} \hat g(v)$, associated to $L$, satisfies $\|g  \|_{\delta, \delta'_0} \lesssim 1,$ for all $\delta>0$. 

%in which $ \sigma (L_{\hat n})$ denotes the spectrum of $L_{\hat n}$. Let $\hat n_0$ and $\lambda_0 \in \sigma (L_{\hat n_0})$ so that $\gamma_0 = \Re\lambda_0$ is positive, and the corresponding eigenfunction $\hat g(v)$ of $L_{\hat n_0}$ satisfies $\| \hat g  \|_{\delta'_0} \lesssim 1,$ for some $\delta'_0>0$; furthermore, there is no open set in $\Omega$ on which $\hat g=0$. In addition, fixing an $n_0 \in \mathbb{Z}^d$ so that $n_0 = |n_0|\hat n_0$, the eigenfunction function $g = e^{i n_0\cdot y} \hat g(v)$, associated to $L$, satisfies $\|g  \|_{\delta, \delta'_0} \lesssim 1,$ for all $\delta>0$. 

%TODO: time scaling....

\item[\bf (H.2)] (\emph{Loss of analyticity for the semigroup (in $y$ only)}) The semigroup $e^{Ls}$, associated to $L$, is well-defined in $X_{\delta, \delta'}$, for $s>0$ and $\delta'>0$ small enough.  Furthermore, there is $\gamma_0 >0$, such that the following holds.
\begin{itemize}

\item For all $\eta>0$, there exist $k_0\in [1, +\infty)$ and  $\lambda_0 \in \Sigma^+$ such that
\begin{equation}
\label{1}
\left| \frac{\Re \lambda_0}{k_0} -  \gamma_0\right| \leq \eta.
\end{equation}
Moreover there exists an eigenfunction $g$ for $L$, associated to $\lambda_0$, such that
$\|g \|_{\delta, \delta_0'}<+\infty $, for all $\delta>0$ and \emph{some} $\delta'_0>0$.%; furthermore, there is no open set in $\Omega$ on which $ g=0$.

\item For any $\Lambda>\gamma_0$, there are $C_\Lambda>0$, $\delta_1'>0$ such that for all $\delta- \Lambda s \geq 0$, $\delta' \in (0,\delta_1']$, and all $\varepsilon>0$,
\begin{equation}
\label{2}
\| e^{Ls} u\|_{\delta- \Lambda s, \delta'} \leq C_\Lambda \| u \|_{\delta, \delta'}, \quad \forall ~u \in X_{\delta, \delta'}.
\end{equation}

\end{itemize}

\item[\bf {\bf (H.3)}] (\emph{Commutator identity}) We have the identity
$$
[\partial_y, L] =0 .
$$

\item[\bf (H.4)] (\emph{Structure of $Q$}) $Q$ is bilinear and we have for all $\delta,\delta'>0$, and all $\varepsilon>0$,
$$
\| Q(f,h)\|_{\delta,\delta'} \leq C_0 \| f \|_{\delta,\delta'} (\| \partial_y h\|_{\delta,\delta'} + \| \partial_z h\|_{\delta,\delta'}), \qquad \forall f,h \in X_{\delta, \delta'},
$$
for some $C_0 >0$.

\item[\bf (H.5)] (\emph{Structure of $R_1$, $R_2$}) $R_1$ is linear and $R_2$ is bilinear. We have for all $\delta,\delta'>0$, and all $\varepsilon>0$,
$$
\| R_1(f)\|_{\delta,\delta'} \leq C_0( \| f \|_{\delta,\delta'} + \| \partial_y f\|_{\delta,\delta'} + \| \partial_z f\|_{\delta,\delta'}), \qquad \forall f \in X_{\delta, \delta'},
$$
$$
\| R_2(f,h)\|_{\delta,\delta'} \leq C_0 \| f \|_{\delta,\delta'} ( \| \partial_y h\|_{\delta,\delta'} + \| \partial_z h\|_{\delta,\delta'}) , \qquad \forall f,h \in X_{\delta, \delta'},
$$
for some $C_0 >0$.

\end{enumerate}

Let us make a few comments about {\bf (H.1)}--{\bf (H.5)}.

The norm $\| \cdot \|_{\delta,\delta'}$ has to be seen as an \emph{analytic} norm used to build a solution to \eqref{e-oscillating}.
The requested properties are classical in the context of spaces of real analytic functions, see e.g. \cite{Gr95}, \cite{MV}. Assumption {\bf (H.1)} yields a violent \emph{instability} for the operator $\mathscr{L}$. In the case $R_1=0$, the interpretation is clear: it reveals that $\mathscr{L}$ has an unbounded unstable spectrum. Indeed it means that for all $\varepsilon>0$,
$$
\mathscr{L} g\left(\frac{\cdot}{\varepsilon}, \cdot\right) = \frac{\lambda_0}{\varepsilon} g\left(\frac{\cdot}{\varepsilon},\cdot\right)
$$
i.e. $g(\frac{\cdot}{\varepsilon},\cdot)$ is an eigenfunction associated to the eigenvalue $\frac{\lambda_0}{\varepsilon}$, with $\Re \lambda_0>0$. In the case where $\mathscr{L}= {L}$ 
%(which holds in the examples studied in Sections \ref{s-Euler} and \ref{s-ill})
, it means that ${L}$ itself has these unstable features.

Assumption {\bf (H.2)} reveals a \emph{loss of analytic regularity} for the semigroup associated to $L$; we emphasize that the loss concerns only the $y$ variable, and not the $z$ variable. The constraint on the admissible losses  $\Lambda>\gamma_0$ in \eqref{2}  is \emph{sharp}, in the sense that it is the best one can hope for, in view of \eqref{1} and its possible consequences on the growth of the spectrum. It means in practice that this number $\gamma_0$ has to be seen as the supremum of some rescaled functional; see Sections \ref{s-Euler} and \ref{s-ill} for illustrations of these facts.
Note also that the eigenfunction $g$ has a very demanding regularity with respect to the first variable. In the context of real analyticity, it means in practice that $g$ has to be very well localized in the Fourier space (with respect to the first variable). 
 The assumption {\bf (H.2)} is certainly the most technical to check in practice, while assumption {\bf (H.3)} is  a simple computation. Note finally that in {\bf (H.4)} and {\bf (H.5)}, the losses of derivatives are only of order $1$, as usual for Cauchy-Kowalevsky type results.

\bigskip

The main result of this section is the following abstract ill-posedness Theorem:

\begin{theorem}\label{thm-abstract}
Assume {\bf (H.1)}--{\bf (H.5)}. 
For all $m,s \in \mathbb{N}$, $\alpha \in (0,1]$, $k \in \mathbb{N}$, there are families of solutions $(U_\varepsilon)_{\varepsilon>0}$ of \eqref{e-abstract}, times $ t_\varepsilon = \mathcal{O}(\varepsilon |\log \varepsilon|)$ and $(x_0,z_0) \in \mathbb{T}^d \times \Omega$ such that
\begin{equation}
\lim_{\varepsilon \to 0} \frac{ \| U_\varepsilon \|_{L^2([0,t_\varepsilon] \times \Omega_\varepsilon)}}{\| \langle z\rangle^m U_\varepsilon|_{t=0} \|_{H^s(\mathbb{T}^d \times \Omega)}}=+\infty
\end{equation}
where $\Omega_\varepsilon =B(x_0, \varepsilon^k) \times B(z_0, \varepsilon^k)$.
\end{theorem}

In the case where the linear differential operator $L$ has constant coefficients, the theorem is due to M\'etivier (\cite{M}) using the power series approach. In applications to the equations we have in mind, the differential operator $L$ typically depends on variables $(x,z)$. Our functional framework is closer to that of Desjardins-Grenier \cite{DG}, who introduced an \emph{analytic} framework for studying nonlinear (Rayleigh-Taylor) \emph{instability}.

%TODO: put here OK ?
\begin{remark}We expect that this abstract framework can be useful to prove  ill-posedness for multi-phase Euler models, see e.g. \cite{Br97}. These models can be (formally) derived in the context of the quasineutral limit of the Vlasov-Poisson equation, see Grenier \cite{Gr95}.      %, in the elliptic regions. 
Whereas the one-dimensional model surely fits the local framework by M\'etivier \cite{M}, the multi-dimensional analogue appears to be nonlocal due to the pressure.%, and this should be treated by our abstract framework. %We leave the details for interested readers. 
 
\end{remark}

The choice of parameters we make below follows M\'etivier \cite{M}. Let $s \in \mathbb{N}$, $\alpha \in (0,1]$, $k \in \mathbb{N}$ be all arbitrary, but fixed. We take $M$ large enough and $\beta>0$ small enough such that
\begin{equation}
\label{def-alpha1}  \alpha' := \frac{M-s}{M} \alpha - \frac{1+2dk}{2M} >0,
\end{equation}
\begin{equation}
\label{e-smallbeta}
\beta M< \frac 12, \qquad \frac{2\beta}{1+ \beta}< \alpha'.
\end{equation}
Let  $\gamma_0>0$ satisfying all requested properties in {\bf (H.2)}. Let $\eta \in \left(0, \min \left(\frac{\gamma_0}{2}, \frac{\beta \gamma_0}{4}  \right)\right)$.
 We obtain a pair of eigenvalue and eigenfunction $(\lambda_0,g)$, and $k_0 \in  [1, +\infty)$, $\delta_0'>0$, such that the first point of {\bf (H.2)} is satisfied. We note that defining
\begin{equation}\label{def-l1}
\gamma_1 = (1+ \beta)\frac{\Re \lambda_0}{k_0},
\end{equation}
 we have  $\gamma_1> \gamma_0$ as well as $\frac{1}{2}\left( \frac{\Re \lambda_0}{k_0} + \gamma_1 \right)>\gamma_0$. We also define
$$ \delta_0 =  \frac{(1-\beta) M}{k_0}  | \log \varepsilon|.$$
Let $(x_0, z_0)$ be such that $g(x_0, z_0) \neq 0$. By continuity, there is $c>0$ such that for all small enough $\varepsilon$,
\begin{equation}\label{pt-g}
|g(x,z) |\ge c,\qquad   \qquad \forall ~(x,z) \in B(x_0, \varepsilon^{k}) \times B(z_0 , \varepsilon^{k}) \subset \Omega_\varepsilon,
\end{equation}
Let us assume that $\lambda_0$ is real (we will explain the general case at the end of the proof).

Unlike M\'etivier \cite{M}, our analysis relies on weighted in time analytic type norms, introduced by Caflisch in his proof of the Cauchy-Kowalevsky theorem (\cite{Caf}; see also \cite{DG}). Precisely, let us introduce the following norm 
\begin{equation}
\begin{aligned}
\| w \| = \sup_{0 \leq \delta \leq \delta_0} \sup_{0 \leq s \leq \frac{1}{\gamma_1}(\delta_0- \delta)} \Big [ 
\| w(s) \|_{\delta, \delta'}  + \left(\delta_0- \delta - \gamma_1 s \right)^{\gamma} \Big(  \| \partial_y w(s) \|_{\delta, \delta'} 
+ M^{-\gamma} |\log \varepsilon|^{-\gamma}  \| \partial_z w(s) \|_{\delta, \delta'} 
  \Big) \Big],
  \end{aligned}
\end{equation}
in which $\delta'$ is a shorthand for $\frac{\delta_0' k_0\delta }{M |\log \varepsilon|}$ and $\gamma$ is an arbitrary fixed number in $(0,1)$. We denote by $X$ the space of functions $w$ such that $\| w \|<+\infty$; it is well known that $(X,\| \cdot \|)$ is a Banach space.

Theorem \ref{thm-abstract} is a consequence of the following lemma, where we construct solutions in $X$ that capture the instability for the scaled system \eqref{e-oscillating}.

\begin{lemma}
\label{l-ill-posedness}
Under the assumptions {\bf (H.1)}--{\bf (H.5)}, there is $\varepsilon_0>0$ such that for all $\varepsilon \in (0,\varepsilon_0]$, 
there exists a solution $u$ of \eqref{e-oscillating} of the form
\begin{equation}
\label{solution}
u(s) = \varepsilon^M e^{\lambda_0s} g + w(s), \qquad \forall ~ s\in [0, s_\varepsilon],
\end{equation}
where $(\lambda_0,g)$ is defined as in {\bf (H.2)}, 
$$s_\varepsilon  = \frac{( {1-\beta} )M}{ k_0 \gamma_1}   |\log \varepsilon|= \frac{1-\beta}{1 + \beta}  \frac{M}{ \lambda_0}   |\log \varepsilon|$$ 
and $w$ is a remainder satisfying 
$$\| w \| \lesssim \varepsilon^{\frac{2\beta }{1+\beta} M}.$$
\end{lemma}

\begin{proof}%[Proof of Lemma~\ref{l-ill-posedness}]
%In this proof, we shall write $\gamma_1 =\Re  \lambda_0 + \nu$, for some $\nu>0$,
We set 
$$u_\mathrm{app}(s) := \varepsilon^M e^{Ls} g = \varepsilon^M e^{\lambda_0s} g.$$
It follows directly from the definition and the assumption $\|g \|_{\delta_0, \delta_0'}\lesssim 1$ that 
$$
\|  u_\mathrm{app} \| \le C \varepsilon^{M} |\log \varepsilon|^\gamma e^{\frac{\lambda_0 \delta_0}{\gamma_1}} \le C \varepsilon^{\kappa} |\log \varepsilon|^\gamma, \qquad \kappa :=\frac{2\beta }{1+\beta}M,
$$
in which the logarithmic loss is due to the weight in time in the norm $\| \cdot \|$. Next, we observe that by definition,
$ u_\mathrm{app} $ is a solution of
$$
\partial_s  u_\mathrm{app}  - L  u_\mathrm{app}  = Q( u_\mathrm{app} , u_\mathrm{app} ) + \varepsilon (R_1(u_\mathrm{app})+ R_2(u_\mathrm{app}, u_\mathrm{app})) + R_\mathrm{app}
$$
where, thanks to the assumptions {\bf (H.4)}-{\bf (H.5)}, the remainder $ R_\mathrm{app} $ satisfies the estimate
$$
\|  R_\mathrm{app}  \| \leq C_\mathrm{app}   \varepsilon^{2\kappa}  |\log \varepsilon|^{2\gamma}.
$$
(Here, we have used the fact that  $\beta M < 1/2$, so that $\varepsilon \le \varepsilon^\kappa$.)
%The integer $N$ will be fixed later.
%
\bigskip

Our goal is now to solve the equation for the difference $w = u - u_\mathrm{app}$: 
\begin{equation}
\label{eq-perturb}
\begin{aligned}
\partial_s w  - L  w  
&= Q( w, w) - Q(u_\mathrm{app},w) - Q(w,u_\mathrm{app})  + \varepsilon(R_1(w)+ R_2(w,w)) 
\\& \quad - \varepsilon (R_2(u_\mathrm{app},w) + R_2(w,u_\mathrm{app})) -  R_\mathrm{app}
\end{aligned}\end{equation}
with $w|_{s=0}=0$. Then $u= w + u_\mathrm{app}$ solves the equation \eqref{e-oscillating} as desired. To that purpose, let us study the following approximation scheme:
$$\begin{aligned}
\partial_s w_{n+1}  - L  w_{n+1}  
&= Q( w_{n}, w_{n}) - Q(u_\mathrm{app},w_{n}) - Q(w_{n},u_\mathrm{app})  + \varepsilon(R_1(w_n)+ R_2(w_n,w_n))  
\\&\quad -  \varepsilon (R_2(u_\mathrm{app},w_n) + R_2(w_n,u_\mathrm{app}))  - R_\mathrm{app}
\end{aligned}$$
with $w_{n+1}|_{s=0}=0$. Set $w_0=0$.

% In what follows, the symbol 
%$A\lesssim B$ means there exists a polynomial $P$ (which does not depend on $n$) such that 
%$A \leq B P(|\log \varepsilon|^\gamma)$. 
We shall prove that for $\varepsilon>0$ small enough, we can make the scheme converge and that the following estimates hold for all $n\geq 0$, % and there is a constant $C_1>0$ such that for all $n\geq 0,
\begin{equation}
\label{rec1}
\| w_n \| \le   C  \varepsilon^{2 \kappa} |\log \varepsilon|^{1+ 2\gamma},
\end{equation}
for some universal constant $C$ (independent of $n$). In addition, for all $n\geq 1$,
\begin{equation}
\label{rec2}
\| w_{n+1}-w_n\| \leq \frac{1}{2} \| w_{n}-w_{n-1}\|.
\end{equation}

By induction, assume that \eqref{rec1} is true for all $k \leq n$. We have
\begin{align*}
w_{n+1}(s) &= \int_0^s e^{(s-\tau)L} \Big [Q( w_{n}, w_{n}) - \Big ( Q(u_\mathrm{app},w_{n}) + Q(w_{n},u_\mathrm{app})   \Big)
\\&\qquad  + \varepsilon \Big (R_1(w_n)+ R_2(w_n,w_n) - R_2(u_\mathrm{app},w_n) - R_2(w_n,u_\mathrm{app} ) \Big)   - R_\mathrm{app}  \Big] \, d\tau \\
&=: I_1(s) + I_2(s) + I_3(s) +  I_4(s).
\end{align*}
We claim that there are $C_0>0$, $\gamma'>0$ and $\varepsilon_0>0$ (independent of $n$) such that for all $\varepsilon\in(0,\varepsilon_0]$, the following estimates hold.

\noindent $\bullet$ For the non-linear term:
$$
\| I_1 \| \leq C_0 \| w_n \|^2 |\log \varepsilon|^{1+\gamma} \leq  C_ 0  \varepsilon^{4 \kappa} |\log \varepsilon|^{1+2\gamma +  \gamma'} . %\leq \frac{C_N}{3} \varepsilon^{N \kappa}.
$$
\noindent $\bullet$  For the linear term:
$$
\| I_2 \| \leq C_0'  \| w_n \| \| u_\mathrm{app}\|  |\log \varepsilon|^{1+\gamma}  \leq   C_ 0 \varepsilon^{1+ 2 \kappa}  |\log \varepsilon|^{1+2\gamma + \gamma' }. %\leq \frac{C_N}{3} \varepsilon^{N \kappa}.
$$
\noindent $\bullet$  For the first remainder:
$$
\| I_3 \| \leq C_0'  \varepsilon  \| w_n \| (1+ \| w_n \|)   |\log \varepsilon|^{1+\gamma}  \leq  C_0 \varepsilon^{1+\kappa} |\log|^{1+2\gamma + \gamma'}.%2 C_ 0' C_\mathrm{app} \varepsilon^{2 \kappa} \varepsilon |\log \varepsilon|^{1+\gamma} %\leq \frac{C_N}{3} \varepsilon^{N \kappa}.
$$
\noindent $\bullet$  For the second remainder:
$$
\| I_4 \| \leq C_\mathrm{app} \varepsilon^{2\kappa} |\log \varepsilon|^{1+ 2\gamma }.
$$
This shows that imposing $\varepsilon$ small enough (but independently of $n$), \eqref{rec1} is satisfied at rank $n+1$, and thus closes the induction argument.

It remains to justify the above estimates. We shall only provide details of the computations for $I_1$, the other ones being similar.
For all 
$\delta \in [1,\delta_0 -\gamma_1 s) $, we set $  \delta' = \frac{\delta_0' k_0 \delta}{M |\log \varepsilon|}$ and $\Lambda := \frac{ \lambda_0}{2k_0} +\frac{\gamma_1}{2}$. Note that $\Lambda> \gamma_0$.
We write $\gamma_1 = \frac{\lambda_0}{k_0} +2 \nu$ with $\nu := \frac{\beta \lambda_0}{2k_0}$. We get, using {\bf (H.2)} and {\bf (H.4)},
\begin{align*}
\| I_1 (s) \|_{\delta, \delta'} &\lesssim   \int_0^s \| e^{(s-\tau)L} Q(w_{n}, w_{n}) \|_{\delta, \delta'}  \, d\tau \\
&\lesssim   \int_0^s \| Q(w_{n}, w_{n}) \|_{\delta + \Lambda(s- \tau), \delta'}  \, d\tau \\
&\lesssim   \int_0^s \| w_{n}(\tau)\|_{\delta + \Lambda(s- \tau), \delta'} \left(\| \partial_y w_{n} (\tau)\|_{\delta + \Lambda(s- \tau), \delta'}  + \| \partial_z w_{n} (\tau)\|_{\delta + \Lambda(s- \tau), \delta'} \right)\, d\tau \\
&\lesssim   \int_0^s \| w_{n}\|^2\left(\delta_0- \delta - \Lambda s - \nu \tau \right)^{-\gamma} (1+ M^{\gamma} |\log \varepsilon|^{\gamma})
  \, d\tau, 
 \end{align*}
 since $\delta' \leq \delta' + \frac{\delta_0' k_0 \Lambda(s- \tau)}{M|\log \varepsilon|}$. We thus get
\begin{align*}
 \| I_1 (s) \|_{\delta, \delta'} &\lesssim  (1+ M^{\gamma} |\log \varepsilon|^{\gamma})\| w_{n}\|^2 \int_0^s  \left(\delta_0- \delta - \Lambda s - \nu \tau \right)^{-\gamma}  
 \, d\tau 
 \\
   &\lesssim  (1+ M^{\gamma} |\log \varepsilon|^{\gamma})\| w_{n}\|^2 \frac{1}{1-\gamma} \frac{1}{\nu}   \left(\delta_0- \delta - \Lambda s  \right)^{1-\gamma}   
\\   &\lesssim  (1+ M^{\gamma} |\log \varepsilon|^{\gamma})  |\log \varepsilon|^{1-\gamma} \| w_{n}\|^2 ,%  \\
% &=: \varepsilon^{2 \kappa} |\log \varepsilon|^{\varphi_\gamma(k)}
\end{align*}
recalling that $\delta_0 = \frac{ (1-\beta) M}{k_0} | \log \varepsilon| $. Likewise, by {\bf (H.3)}, {\bf (H.2)} and {\bf (H.4)}, we obtain 
\begin{align*}
\| \partial_y I_1(s) \|_{\delta, \delta'} &\lesssim   \int_0^s \| \partial_y  e^{(s-\tau)L} Q(w_{n}, w_{n}) \|_{\delta, \delta'}  \, d\tau \\
&\lesssim   \int_0^s \| e^{(s-\tau)L} \partial_y Q(w_{n}, w_{n}) \|_{\delta, \delta'}  \, d\tau \\
&\lesssim   \int_0^s \|\partial_y [Q(w_{n}, w_{n})] \|_{\delta + \Lambda(s- \tau), \delta'}  \, d\tau \\
&\lesssim   \int_0^s \frac{2\delta_0}{(\delta_0- \delta - \Lambda s -\nu\tau )}\|Q(w_{n}, w_{n}) \|_{\frac{\delta_0 - \gamma_1 \tau}{2}+\frac{\delta + \Lambda(s- \tau)}{2}, \delta'}  \, d\tau 
\\
&\lesssim   \int_0^s \| w_{n}\|^2 \delta_0 \left(\delta_0- \delta - \Lambda s - \nu \tau \right)^{-1-\gamma} (1+M^\gamma |\log \varepsilon|^{\gamma}) \, d\tau
\\
&\lesssim   \| w_{n}\|^2 |\log \varepsilon|^{1+\gamma}  \left(\delta_0- \delta - \Lambda s - \nu s\right)^{-\gamma} .
 \end{align*}
Note that $\Lambda + \nu = \gamma_1$. Consequently,  
we deduce 
\begin{equation*}
\| \partial_y I_1 (s) \|_{\delta, \delta'}  \lesssim   \| w_{n}\|^2 |\log \varepsilon|^{1+\gamma} \left(\delta_0- \delta - \gamma_1 s   \right)^{-\gamma} .
% \varepsilon^{4 \kappa} |\log \varepsilon|^{2(1+\gamma) +1 }\left(\delta_0- \delta - \gamma_1 t   \right)^{-\gamma} 
 \end{equation*}
Now we use that there is no loss in the $z$ variable for the semigroup. Recalling  $  \delta' = \frac{\delta_0' k_0\delta}{M |\log \varepsilon|}$ and $\gamma_1 = \Lambda + \nu$, we get with similar computations
\begin{align*}
\| \partial_z I_1 (s) \|_{\delta, \delta'} &\lesssim   \int_0^s \| \partial_z e^{(s-\tau)L}  Q(w_{n}, w_{n}) \|_{\delta, \delta'}  \, d\tau \\
&\lesssim   \int_0^s \|e^{(s-\tau)L}[Q(w_{n}, w_{n})] \|_{\delta, \frac{\delta_0' k_0}{M|\log \varepsilon|} \left[\frac{\delta_0 - \gamma_1 \tau}{2}  + \frac{\delta}{2}\right]}   \delta_0\left(\delta_0- \delta - \gamma_1 \tau \right)^{-1}  \, d\tau  
 \\
&\lesssim \int_0^s \|Q(w_{n}, w_{n}) \|_{\delta+ \Lambda(s-\tau), \frac{\delta_0' k_0 }{M|\log \varepsilon|} \left[\frac{\delta_0 - \gamma_1 \tau}{2}  + \frac{\delta}{2}\right]}  \delta_0\left(\delta_0- \delta - \Lambda \tau - \nu \tau \right)^{-1} \, d\tau \\
&\lesssim \int_0^s \|Q(w_{n}, w_{n}) \|_{\delta+ \Lambda(s-\tau), \frac{\delta_0' k_0 }{M|\log \varepsilon|} \left[\frac{\delta_0 - \gamma_1 \tau}{2}  + \frac{\delta}{2}\right]}  \delta_0\left(\delta_0- \delta - \Lambda s - \nu \tau \right)^{-1} \, d\tau \\
&\lesssim   \int_0^s  \| w_{n}\|^2  \frac{\delta_0 (1+ M^\gamma|\log \varepsilon|^\gamma)}{(\delta_0- \delta - \Lambda s -\nu\tau )^{\gamma +1}}  \, d\tau \\
&\lesssim \| w_{n}\|^2    |\log \varepsilon|^{1+\gamma}  \left(\delta_0- \delta - \gamma_1 s   \right)^{-\gamma}.
%&\lesssim   \int_0^t \| w_{n}\| \| w_{n} \| \left(\delta_0- \delta - \Lambda t - \nu \tau \right)^{-\gamma-1} (1+ |\log \varepsilon|^{-\gamma})
 \end{align*}
We end up with the claimed estimate for $\|I_1\|$, using the induction assumption on $\| w_n\|$. The contraction estimates are now straightforward; we have indeed for all $n \geq 2$,
$$
\| w_{n+1}- w_{n} \| \leq C_1 |\log \varepsilon|^{1+\gamma} ( \|w_{n-1}\| + \|w_{n}\|) \|\| w_{n}- w_{n-1} \|   + C_2 \varepsilon^\kappa |\log \varepsilon|^{1+\gamma} \|\| w_{n}- w_{n-1} \|.
$$
By \eqref{rec1}, we have
$$
 \|w_{n-1}\| + \|w_{n}\| \lesssim \varepsilon^{2\kappa} |\log \varepsilon|^{1+2\gamma},
$$
so that by imposing $\varepsilon>0$ small, we obtain \eqref{rec2}.

Since $(X,\| \cdot \|)$ is a Banach space, we get a solution $w$ of \eqref{eq-perturb}, satisfying in addition the estimate
$$
\| w \| \leq  C\varepsilon^{2\kappa} |\log \varepsilon|^{1+2\gamma}.
$$
\end{proof}

Finally, we complete the proof of Theorem~\ref{thm-abstract}: 
\begin{proof}[Proof of Theorem~\ref{thm-abstract}]
Consider the initial condition $U_\varepsilon|_{t=0}(x,z)  = \varepsilon^M g\left( \frac{x}{\varepsilon}, z \right)$ for \eqref{e-abstract}.
We note that
$$
\| \langle z \rangle^m U_\varepsilon|_{t=0} \|^\alpha_{H^s} \leq C' \varepsilon^{\alpha(M-s)}
$$
By Lemma~\ref{l-ill-posedness}, we obtain a solution $ u_\varepsilon(s,y,z)$ of \eqref{e-oscillating} with initial condition $\varepsilon^M g(y,z)$, satisfying~\eqref{solution}. Thus we get a solution $U_\varepsilon(t,x,z) =  u_\varepsilon \left(\frac{t}{\varepsilon}, \frac{x}{\varepsilon}, z\right)$ for \eqref{e-abstract}. Consider the time
$$s_\varepsilon  = \frac{1-\beta}{1 + \beta} \frac{M}{\lambda_0} |\log \varepsilon|$$
and let $t_\varepsilon = \varepsilon s_\varepsilon$. Using the embedding in $L^2$, \eqref{solution}, and the lower bound \eqref{pt-g}, we have, for some constants $\theta_1, \theta_1'>0$ that are independent of $\varepsilon$:
$$
\begin{aligned}
\| U_\varepsilon  \|_{ L^2([0,t_\varepsilon]\times \Omega_\varepsilon )} &=  \|\varepsilon^M e^{L t/\varepsilon} g (\frac{\cdot}{\varepsilon}, \cdot)+ w  (\frac{\cdot}{\varepsilon}, \cdot)\|_{L^2([0,t_\varepsilon]\times  \Omega_\varepsilon )} \\
 &\geq  \|\varepsilon^M e^{Lt/\varepsilon} g(\frac{\cdot}{\varepsilon}, \cdot) + w (\frac{\cdot}{\varepsilon}, \cdot)  \|_{L^2([t_\varepsilon - \varepsilon,t_\varepsilon]\times B(x_0, \varepsilon^{k}) \times B(z_0, \varepsilon^{k}) )} \\
&\geq \theta_1 \varepsilon^{\kappa} \varepsilon^{dk +\frac 12}   -  \| w  (\frac{\cdot}{\varepsilon}, \cdot)\|_{L^2( [t_\varepsilon - \varepsilon,t_\varepsilon] \times B(x_0, \varepsilon^{k}) \times B(z_0, \varepsilon^{k}) )}\\
&\geq \theta_1\varepsilon^{\kappa} \varepsilon^{dk+\frac 12}  -   C \varepsilon^{dk+\frac 12} \varepsilon^{2\kappa} |\log \varepsilon|^{\gamma_1}\\
&\geq \theta_1'\varepsilon^{\kappa} \varepsilon^{dk+\frac 12} 
\end{aligned}
$$
for sufficiently small $\varepsilon$. We recall that $\kappa = \frac{2\beta}{1+\beta} M$. By a view of the choice on the parameters in \eqref{def-alpha1} and \eqref{e-smallbeta}, we get 
$$
\frac{\| U_\varepsilon \|_{ L^2([0,t_\varepsilon]\times \Omega_\varepsilon )} }{\|\langle z \rangle^m U_\varepsilon|_{t=0} \|^\alpha_{H^s}} \geq \frac{\theta'_1}{C'}\varepsilon^{M\left(\frac{2\beta}{1+ \beta}-\alpha'\right)},
$$
which tends to infinity as $\varepsilon \to 0$. This concludes the proof of the theorem when $\lambda_0$ is real.
In the general case, the modifications are standard, see e.g. M\'etivier \cite{M}. In Lemma~\ref{l-ill-posedness}, one has to replace~\eqref{solution} by
\begin{equation}
u(s) = \varepsilon^M \Re (e^{\lambda_0s} g) + w(s), \qquad \forall ~ s\in [0, s_\varepsilon],
\end{equation}
meaning that instead of comparing to a pure exponentially growing mode,  we have to compare $u$ to an exponentially growing mode multiplied by an oscillating function. The above analysis can be performed again, making sure to avoid the (discrete) times when this oscillating function cancels.

\end{proof}

\section{Ill-posedness of the hydrostatic Euler equations}\label{s-Euler}

In this section, we give the proof of Theorem \ref{t-hydro}, establishing the ill-posedness of the hydrostatic Euler equations \eqref{hydro}, which we write in the stream-vorticity formulation: 
\begin{equation}\label{vort00}
\partial_t  \Omega + U \partial_x  \Omega + V \partial_z  \Omega = 0,
\end{equation}
in which $U := \partial_z  \Phi$, $ V := -\partial_x  \Phi$, and the stream function $ \Phi$ solves the elliptic problem: 
\begin{equation}\label{def-stream}
\partial_z^2  \Phi =  \Omega , \qquad  \Phi_{\vert_{z=\pm 1}} = 0.
\end{equation}

We shall prove in this section how Theorem \ref{t-hydro} follows from our abstract ill-posedness framework. We work with the analytic function space $X_{\delta, \delta'}$, equipped with the following norm: 
\begin{equation}\label{vort-norm0}
 \| \omega\|_{\delta, \delta'}: =\sum_{n\in \mathbb{Z}} \sum_{k\ge 0} \|\partial_z^k \omega_n \|_{L^2([-1,1])} \frac{|\delta'|^{k }}{ k!} e^{\delta |n| } ,\end{equation}
for any $\delta, \delta'>0$, in which $\omega_n = \langle \omega, e^{inx}\rangle_{L^2(\mathbb{T})}$ stands for the Fourier coefficients of $\omega$ with respect to $x$-variable. We also denote by $X_{\delta'}$ the $z$-analytic function space, equipped with the following norm:  
\begin{equation}\label{vort-norm0z}  \| \omega\|_{\delta'}: = \sum_{k\ge 0} \|\partial_z^k \omega \|_{L^2([-1,1])} \frac{|\delta'|^{k }}{ k!} .\end{equation}

We study ill-posedness near a well-chosen shear flow, i.e. $ \Omega = U'(z)$. 
We assume that $U$ is real analytic so that the following norm 
\begin{equation}\label{U-analytic}
 ||| U |||_{\delta'}: = \sum_{k\ge 0} \|\partial_z^k U' \|_{L^\infty([-1,1])} \frac{|\delta'|^{k }}{ k!} 
 \end{equation}
 is finite,  for some $\delta'>0$. 

By a view of the analytic norms \eqref{vort-norm0}-\eqref{vort-norm0z}, we have for all $0<\delta'<\delta_1'$,
$$ \| \partial_z\omega\|_{\delta'} \le \frac{\delta_1'}{\delta_1' - \delta'} \| \omega\|_{\delta_1'},$$
and so 
$$ \| \partial_z \omega\|_{\delta, \delta'} \le \frac{\delta_1'}{\delta_1' - \delta'} \| \omega\|_{\delta, \delta_1'}.$$
We can argue similarly for $\| \partial_y\omega\|_{\delta, \delta'}$.%,  the assumption {\bf (H.1)} follows.

We write the perturbed solution in the fast variables as follows: 
$$
 \Omega(t,x,z) =U'(z) +  \omega \left(\frac t\varepsilon, \frac x\varepsilon, z\right), \qquad \Phi(t,x,z) = \int_{-1}^z U(\theta)\; d\theta + \varphi\left(\frac t\varepsilon, \frac x\varepsilon, z\right).$$
Let $(s,y) =  (\frac t\varepsilon, \frac x\varepsilon)$. The function $\omega(s,y,z)$ solves 
\begin{equation}\label{pert-vort}
\partial_s \omega - \mathcal{L}\omega = - \partial_z \varphi \partial_y \omega + \partial_y \varphi \partial_z \omega,
\end{equation}
in which the linearized operator is defined by 
\begin{equation}\label{Edef-cL}
\mathcal{L} \omega: = - U \partial_y \omega  + \partial_y \varphi \,  U'' , \qquad \partial_z^2 \varphi = \omega, \qquad \varphi_{\vert_{z=\pm 1}} = 0.
\end{equation}

To treat the loss of derivatives from each quantity in the quadratic term $\partial_y \varphi \partial_z \omega$, we further write the above equation in a matrix form.  Set 
 \begin{equation}\label{def-L-E} 
 L: = \begin{pmatrix} \mathcal{L} &0 &0 \\ 0 &  \mathcal{L} &0\\ 0& -U' + U'' \partial_z \varphi(\cdot) + U''' \varphi(\cdot) & - U \partial_y \end{pmatrix} ,\end{equation}
in which by convention, $\varphi(W)$ solves $\partial_z^2 \varphi(W) = W$ with $\varphi_{\vert_{z=\pm 1}} = 0$. Similarly, for any two vector fields $V = (v_1, v_2, v_3)$ and $W = (w_1, w_2, w_3)$, we set 
$$ Q(V,W): = \begin{pmatrix} - \partial_z \varphi (v_1) w_2 + \varphi(v_2) w_3 \\ - \partial_z \varphi (v_2)  w_2 - \partial_z \varphi (v_1) \partial_y w_2 + \partial_y \varphi(v_2) w_3 + \varphi(v_2) \partial_y w_3\\  -v_1 w_2 - \partial_z \varphi (v_1) \partial_z w_2 + \partial_z \varphi(v_2) w_3 + \varphi(v_2) \partial_z w_3 \end{pmatrix} .$$

It follows that $\omega$ solves \eqref{pert-vort} if and only if the function 
$$ W: = \begin{pmatrix} \omega \\ \partial_y \omega \\ \partial_z \omega \end{pmatrix} $$
solves 
\begin{equation}\label{pert-vortM}
\partial_s W - LW= Q(W,W). 
\end{equation}

We shall show the ill-posedness of \eqref{pert-vortM} by directly checking the assumptions {\bf (H.1)}--{\bf (H.5)} made in our abstract ill-posedness framework.

 \subsection{Unbounded unstable spectrum of the linearized operator}
 
Our starting point is the work by Renardy \cite{Renardy}, in which he showed that the linearized hydrostatic Euler system near certain shear flows $U(z)$ possesses ellipticity or unbounded unstable spectrum.
Indeed, let us study the linearization near $U(z)$:
\begin{equation}\label{vort-lin}
\partial_t \omega + U \partial_y \omega  - \partial_y   \varphi \,  U'' = 0, \qquad \partial_z^2 \varphi = \omega, \qquad \varphi_{\vert_{z=\pm 1}} = 0,
\end{equation}
and search for a growing mode of the form 
\begin{equation}
\label{grow-H}
\omega = e^{in (y-ct)} \hat \omega(z),
\end{equation}
with $\Im c \neq 0$ and $\hat \omega = \partial_z^2 \hat \varphi$. The stream function $\hat \varphi$ then solves the  Rayleigh problem:  
\begin{equation}\label{Rayleigh} (U-c)\partial_z^2\hat  \varphi - U'' \hat \varphi = 0, \qquad \hat \varphi_{\vert_{z=\pm 1}} = 0.\end{equation}
This is a very classical problem in fluid mechanics (see for instance the recent work \cite{GGN}). There are two independent solutions of the Rayleigh problem: 
\begin{equation}\label{hom-phi} \hat \varphi_1 = U - c, \qquad \hat \varphi_2 = (U-c) \int_{-1}^z \frac{1}{(U(z')-c)^2}\; dz',\end{equation}
whose Wronskian determinant is $W[\hat \varphi_1, \hat \varphi_2] = \partial_z \hat \varphi_2 \hat \varphi_1 - \partial_z \hat \varphi_1 \hat \varphi_2 = 1$. The pair $(\hat \varphi, c)$ solves the Rayleigh problem if and only if $c$ is a zero of the (Evans) function
\begin{equation}
\label{evans}
 D(c) : = \hat \varphi_1 (-1) \hat \varphi_2(1) - \hat \varphi_1 (1) \hat \varphi_2(-1).
 \end{equation}
This precisely means that $c$ has to solve the equation 
\begin{equation}\label{un-condR} \int_{-1}^1 \frac{1}{(U(z) - c)^2}\; dz =0,\end{equation}
(see also \cite[Theorem 1]{Renardy}). As an explicit example, one can take $U(z) = \mathrm{tanh}(\frac z{d_1})$ for small $d_1>0$ as shown in \cite{CM}. Since $c$ does not depend on $n$, the unstable spectrum is unbounded. Let us summarize this discussion in the following statement:

\begin{lemma}\label{lem-unstable} The linearized operator $\mathcal{L}$  possesses a growing mode of the form \eqref{grow-H} if and only if $c$ is a zero of the Evans function~\eqref{evans}. If such a growing mode exists, the unstable spectrum is unbounded, containing all the points $\lambda = - i n c$, with $n \in \mathbb{Z}$ such that $n \Im c>0$ and with corresponding eigenfunctions of the form
\begin{equation}
\label{eq:eig}
\begin{aligned}
 %\hat \varphi_0&= - \frac{\hat \varphi_2(1)}{\hat \varphi_1(1)}  \hat \varphi_1 +  \hat \varphi_2, \\
\omega &= e^{in y}\partial_z^2 \hat \varphi_2.
\end{aligned}
\end{equation}
\end{lemma}

\subsection{Sharp semigroup bounds}

Let $L$ be the matrix operator defined as in \eqref{def-L-E}. From now on, we consider a shear flow $U(z)$ such that $ ||| U |||_{\delta'_1} < +\infty$ for some $\delta'_1>0$, with which the unstable spectrum of the linearization \eqref{vort-lin} is unbounded; see the previous section. We set $\gamma_0$ to be defined by 
\begin{equation}\label{def-g0} \gamma_0 : = \max_{\Im c \neq 0} \Big\{ \Im c ~:~ D(c) = 0\Big\}.\end{equation}
The above maximum exists and is positive, since by a view of~\eqref{un-condR} the Rayleigh problem has no solution as $|c|\to \infty$ and $D(c)$ is continuous (in fact, analytic) in  $\{\Im c\neq 0\}$. Let $c_0$ be the solution of $D(c_0) =0$ so that $\gamma_0 = \Im c_0$, and let $\omega$ be the corresponding eigenfunction as in \eqref{eq:eig}. By Lemma \ref{lem-unstable}, $\lambda_0 = -i n c_0$ is an unstable eigenvalue of $\mathcal{L}$, for all $n \in \mathbb{Z}$, so that $\Re \lambda_0 = n \Im c_0>0$. Since $\omega=e^{in y}\partial_z^2 \hat \varphi_2$, as defined in \eqref{hom-phi} , the regularity of $\omega$ follows from that of the given shear flow $U(z)$.

\begin{proposition}\label{prop-semigroup-cL-e} Let $\delta, \delta' >0$ and let $\gamma_0$ be defined as in \eqref{def-g0}. The semigroup $e^{Ls}$, associated to $L$, is well-defined in $X_{\delta, \delta'}$, for small $s>0$ and small $\delta'$. More precisely, for any $\gamma > \gamma_0$, there is a positive constant $C_\gamma$ so that 
$$ \| e^{Ls} h \|_{\delta - \gamma s, \delta'} \le C_\gamma  \| h\|_{\delta, \delta'},$$
for all $h \in X_{\delta, \delta'}$, $0<\delta' \ll  \gamma_0$, and for all $s$ so that $\delta - \gamma s> 0$.  
\end{proposition}

We start the proof of Proposition~\ref{prop-semigroup-cL-e} by proving the same bounds for the semi-group $e^{\mathcal L s}$. In Fourier variables, we first study the semigroup $\omega = e^{\mathcal{L}_n s} h$, solving 
$$\partial_s \omega  - \mathcal{L}_n \omega  =0, \qquad \omega_{\vert_{s=0}} = h,$$
with $$\mathcal{L}_n \omega : = - in U \omega  + in \varphi U'' , \qquad \partial_z^2 \varphi = \omega, \qquad \varphi_{\vert_{z=\pm 1}} = 0.$$

\begin{lemma}\label{lem-trans} For each $n \in \mathbb{Z}$, let $\omega$ solve the transport equation $(\partial_s + inU)\omega  = g$. Then, 
$$ \| \omega(s)\|_{\delta'} \le e^{\delta' |n| ||| U |||_{\delta'} s}  \| \omega(0)\|_{\delta'} + \int_0^s e^{\delta' |n| ||| U |||_{\delta'} (s-\tau)} \| g(\tau)\|_{\delta'}\; d\tau,$$
for all $\delta' \in (0, \delta'_1]$. 
\end{lemma}
\begin{proof} The proof follows by $L^2$ energy estimates. Indeed, differentiating the equation for $\omega$, we have 
$$ \partial_s \partial_z^k \omega = - i n \sum_{j=0}^k \frac{k!}{j! (k-j)!} \partial_z^{k-j}U \partial_z^j \omega  +  \partial_z^k g.$$ 
Now taking the $L^2$ product against $\partial_z^k \omega$ and taking the real part, we get  
$$\frac12 \frac{d}{ds} \|\partial_z^k \omega\|^2_{L^2(-1,1)} \le \Big[ |n|  \sum_{j=0}^{k-1} \frac{k!}{j! (k-j)!} \| \partial_z^{k-j}U \partial_z^j \omega\|_{L^2(-1,1)}  + \|\partial_z^k g\|_{L^2(-1,1)} \Big] \|\partial_z^k \omega\|_{L^2(-1,1)},$$  
upon noting that the real part of $i \langle U \partial_z^k \omega, \partial_z^k \omega\rangle_{L^2(-1,1)}$ is equal to zero. This yields 
$$\frac{d}{ds} \|\partial_z^k \omega\|_{L^2(-1,1)} \le |n|  \sum_{j=0}^{k-1} \frac{k!}{j! (k-j)!} \| \partial_z^{k-j}U\|_{L^\infty(-1,1)} \|\partial_z^j \omega\|_{L^2(-1,1)}  + \|\partial_z^k g\|_{L^2(-1,1)}.$$  
By definition of the analytic norm, we get from the above estimate
\begin{equation}\label{a-transport}\begin{aligned}
\frac{d}{ds} \| \omega\|_{\delta'}
&\le |n|\sum_{k\ge 0}\frac{|\delta'|^k}{k!}\sum_{j=0}^{k-1} \frac{k!}{j! (k-j)!} \| \partial_z^{k-j}U\|_{L^\infty(-1,1)} \|\partial_z^j \omega\|_{L^2(-1,1)}  +  \| g\|_{\delta'}
\\&\le \delta'|n|\sum_{k\ge 0}\sum_{j=0}^{k-1} \frac{|\delta'|^j}{j! }\frac{|\delta'|^{k-j-1}}{(k-j-1)! } \| \partial_z^{k-j-1}U'\|_{L^\infty(-1,1)}  \|\partial_z^j \omega\|_{L^2(-1,1)}  + 
 \| g\|_{\delta'}
\\&\le \delta'|n|\sum_{\ell \ge 0}\sum_{j\ge 0} \frac{|\delta'|^j}{j! }\frac{|\delta'|^{\ell}}{\ell ! } \| \partial_z^{\ell}U'\|_{L^\infty(-1,1)}  \|\partial_z^j \omega\|_{L^2(-1,1)}  + 
 \| g\|_{\delta'}
 \\& \le  \delta' |n| \, |||U'|||_{\delta'} \| \omega\|_{\delta'} + \| g\|_{\delta'}.
 \end{aligned}\end{equation}
The lemma follows from the Gronwall inequality.  
\end{proof}

\begin{lemma}\label{lem-semigroupLn-e} Let $\gamma_0$ be defined as in \eqref{def-g0}. For each $n \in \mathbb{Z}$, the operator $\mathcal{L}_n$ generates a continuous semigroup $e^{\mathcal{L}_n s}$  from $X_{\delta'}$ to itself, for small $ \delta'>0$.  In addition, for any $\gamma>\gamma_0$, there is a positive constant $C_\gamma$ so that 
$$ \| e^{\mathcal{L}_n s} h \|_{\delta'} \le C_\gamma e^{\gamma|n|s} \| h\|_{\delta'},\qquad \forall~s\ge 0,$$
for all $h \in X_{\delta'}$, and $0<\delta'\ll \gamma_0$. 
\end{lemma}

\begin{proof}

By the time rescaling $s \mapsto s |n|$, it suffices to study the semigroup for $n=\pm 1$.  Let us focus on  $n=1$, the case $n=-1$ being identical.  We obtain the sharp bound via the inverse Laplace transform for the semigroup  (\cite{Pazy} or \cite[Appendix A]{Zum}): 
$$ e^{\mathcal{L}_1 s} h = \mathrm{PV} \,  \frac{1}{2\pi i} \int_{\gamma-i\infty}^{\gamma+ i \infty} e^{\lambda s} (\lambda - \mathcal{L}_1)^{-1}h \; d\lambda,$$
for sufficiently large $\gamma$, where $\mathrm{PV}$ denotes the Cauchy principal value.
Set $\hat \omega : = (\lambda - \mathcal{L}_1)^{-1}h $. We shall solve the resolvent equation for all $\lambda = -ic$ with sufficiently large values of $|\Im c|$. It follows that 
$$ \hat \omega = \partial_z^2 \hat \varphi, \qquad  (U-c)\partial_z^2\hat  \varphi - U'' \hat \varphi = \frac{h}{i}, \qquad \hat \varphi_{\vert_{z=\pm 1}} = 0.$$
This is a nonhomogenous Rayleigh problem with unknown $\hat \varphi$. By definition  of $\gamma_0$, see \eqref{def-g0}, for all $c$ such that $|\Im c| > \gamma_0$, the Rayleigh operator $(U-c)\partial_z^2 - U'' $ is invertible, and so the Rayleigh problem has an unique solution $\hat \varphi$ (in fact, one can derive an explicit representation involving the homogenous solutions $\hat \varphi_1, \hat \varphi_2$ defined as in \eqref{hom-phi}, see e.g. \cite{GGN}). In addition, together with zero boundary conditions, there holds 
$$ \| \hat \varphi\|_{H^2} \le \frac{C}{1+|\Re c|} \| h\|_{L^2} , \qquad \forall~c\in \mathbb{C}~:~|\Im c| >\gamma_0.$$
Recalling that $\hat \omega = (\lambda - \mathcal{L}_1)^{-1} h$ is the resolvent solution with $\hat \omega = \partial_z^2 \hat \varphi$, we then get 
$$
\begin{aligned}
 e^{\mathcal{L}_1 s} h 
 &= \mathrm{PV} \,  \frac{1}{2\pi i} \int_{\gamma-i\infty}^{\gamma+ i \infty} e^{\lambda s} (\lambda - \mathcal{L}_1)^{-1}h \; d\lambda,
 \\
 &= \mathrm{PV} \,  \frac{1}{2\pi } \int_{i\gamma-\infty}^{i\gamma+  \infty} e^{-i cs} \Big[ \frac{U''}{U - c} \hat \varphi + \frac{h}{i(U-c)} \Big] \; dc
 \\
 &= \mathrm{PV} \,  \frac{1}{2\pi } \int_{i\gamma-\infty}^{i\gamma+  \infty} e^{-i cs} \frac{U''}{U - c} \hat \varphi  \; dc + e^{-i U(z) s} h.
\end{aligned} $$
This identity, together with the elliptic bound on $\hat \varphi$, and the fact that we can take any $c$ such that $\Im c=\gamma> \gamma_0$ yield%$\gamma$ to be arbitrary as long as $\gamma>\gamma_0$ yield 
$$ 
\begin{aligned}\| e^{\mathcal{L}_1 s} h\|_{L^2} &\le   \frac{C}{2\pi } \int_\mathbb{R} e^{\gamma s} \frac{\|U''\|_{L^\infty}}{ \sqrt{|U-\Re c|^2 + \gamma^2}} \|\hat \varphi \|_{L^2}  \; d (\Re c) + \| h\|_{L^2}  \\
 &\le   \frac{C}{2\pi } \int_\mathbb{R} e^{\gamma s} \frac{\|U''\|_{L^\infty}}{ \sqrt{|U-\Re c|^2 + \gamma^2}} \frac{1}{1+|\Re c|} \|h \|_{L^2}  \; d (\Re c) + \| h\|_{L^2}  \\
&\le C_\gamma e^{\gamma s} \| h\|_{L^2},
\end{aligned}$$
for all $\gamma>\gamma_0$.

Next, we derive analytic estimates for $\omega := e^{\mathcal{L}_1 s} h$, solving 
$$ (\partial_s + iU)\omega  - iU''\varphi = 0, \qquad \partial_z^2 \varphi = \omega,$$  
with zero boundary conditions on $\varphi$. 
%For convenience, we also introduce 
%\begin{equation}\label{vort-normz}
% \| \omega\|_{\delta'}: = \sum_{k\ge 0} \|\partial_z^k \omega \|_{L^2([-1,1])} \frac{|\delta'|^{k} }{k!} 
% \end{equation}
%for any $\delta'>0$. We note that $\| fg\|_{\delta'} \le C\| f\|_{\delta'} \| g\|_{\delta'}$. 
Using \eqref{a-transport}, we have 
$$\begin{aligned}
\frac{d}{ds} \| \omega\|_{\delta'} \le \delta'  ||| U |||_{\delta'} \| \omega\|_{\delta'}  + \| U'' \varphi\|_{\delta'}.
 \end{aligned}$$
By definition, $\| U'' \varphi \|_{\delta'}  \leq  \| U''  \|_{\delta'} \| \varphi \|_{\delta'}$ and
$\| \varphi\|_{\delta'}\le C_0 \| \varphi\|_{H^1(-1,1)} + |\delta'|^2\| \partial_z^2 \varphi\|_{\delta'}$. We thus obtain
$$\begin{aligned}
\frac{d}{ds} \| \omega\|_{\delta'} 
&\le C_0\| \varphi\|_{H^1(-1,1)} + \delta' \Big( ||| U |||_{\delta'} + \delta' \| U''\|_{\delta'}\Big) \| \omega\|_{\delta'}
\\
&\le C_0 \|  \omega  \|_{L^2} + \delta' \Big( ||| U |||_{\delta'} + \delta' \| U''\|_{\delta'}\Big) \| \omega\|_{\delta'} ,
 \end{aligned}$$
in which the last estimate is due to the Poincar\'e inequality. 
Hence, for $\delta'$ sufficiently small so that $\delta' (|||U'|||_{\delta'} + \delta' \| U''\|_{\delta'} ) \le \gamma_0<\gamma $, the Gronwall inequality  and the previous $L^2$ bound give 
$$ 
\begin{aligned}
\| \omega(s)\|_{\delta'} &\le e^{\delta'( ||| U |||_{\delta'} + \delta' \| U''\|_{\delta'}) s} \| h\|_{\delta'}+ C_0 \int_0^s e^{\delta'( ||| U |||_{\delta'} + \delta' \| U''\|_{\delta'}) (s-\tau)} \| \omega(\tau)\|_{L^2} \; ds 
\\
&\le e^{\gamma s} \| h\|_{\delta'}+ C_\gamma e^{\gamma s} \| h\|_{L^2},
\end{aligned}$$  
which proves the claimed bound for $e^{\mathcal{L}_1 s}$. 

\end{proof}

We are now in position to prove Proposition~\ref{prop-semigroup-cL-e}.

\begin{proof}[Proof of Proposition~\ref{prop-semigroup-cL-e}]
Let us prove the bound for the semigroup $(W_1,W_2,W_3) = e^{Ls}H$, with $L$ being the matrix operator defined as in \eqref{def-L-E}. Let $W_{j,n}(z) = \langle W_j(\cdot, z), e^{in y}\rangle_{L^2(\mathbb{T})}$, for $j = 1,2,3$. By the structure of the matrix operator $L$, Lemma~\ref{lem-semigroupLn-e} gives the bound for $W_{1,n}, W_{2,n}$:   
\begin{equation}\label{W2-bd} \| W_{j,n}  \|_{\delta'} \le C_\gamma e^{\gamma  |n|s} \| H_{j,n}\|_{\delta'},\qquad j=1,2,\end{equation}
for all $\gamma >\gamma_0$, $\delta' \le \gamma_0$, and all $n \in \mathbb{Z}$. Hence, by definition of the analytic norm, we get 
$$ \| W_j (s)\|_{\delta - \gamma s, \delta'} \le  C_\gamma \sum_{n\in \mathbb{Z}} e^{\gamma |n|s} \| H_{j,n}\|_{\delta'}  e^{(\delta - \gamma s) |n| } \le C_\gamma \| H_j\|_{\delta , \delta'}.$$
for any $s$ so that $\delta - \gamma s >0$. 

As for $W_3$, we write 
$$( \partial_s + U \partial_y )W_3  + U' W_2 - U'' \partial_z \varphi(W_2) - U''' \varphi(W_2)  = 0.$$ 
Again, in Fourier variables, we have
$$( \partial_s + i n U )W_{3,n}  + U' W_{2,n} - U'' \partial_z \varphi(W_{2,n}) - U''' \varphi(W_{2,n})  = 0.$$ 
Again, using \eqref{a-transport}, we get 
$$\begin{aligned}
\frac{d}{ds} \|W_{3,n}\|_{\delta'} 
&\le  \delta' |n| ||| U'|||_{\delta'} \| W_{3,n}\|_{\delta'} +  \| U' W_{2,n} - U'' \partial_z \varphi(W_{2,n}) - U''' \varphi(W_{2,n}) \|_{\delta'}
 \\& \le \delta' |n| ||| U |||_{\delta'} \| W_{3,n}\|_{\delta'}+C_0 \| W_{2,n}\|_{\delta'} .
 \end{aligned}$$
Using the Gronwall inequality, the bound \eqref{W2-bd} on $W_{2,n}$, and the assumption that $\delta'  ||| U |||_{\delta'}  \le \gamma_0$, we obtain 
$$ \| W_{3,n}(s)\|_{\delta'} \le C_\gamma e^{\gamma |n| s} \| H_n\|_{\delta'}.$$
Summing  the estimate over $n \in \mathbb{Z}$ yields the claimed bound for $W_3$. This completes the proof of the semigroup estimate and thus of the proposition. 
\end{proof}

\subsection{Conclusion}

We now show that the system \eqref{pert-vortM} fits our abstract framework. 
%
%First, we recall the analytic norm:
%\begin{equation*}%\label{vort-norm}
% \| \omega\|_{\delta, \delta'}: =\sum_{n\in \mathbb{Z}} \sum_{k\ge 0} \|\partial_z^k \omega_n \|_{L^2([-1,1])} \frac{|\delta'|^{k }}{ k!} e^{\delta |n| } ,\end{equation*}
%for any $\delta, \delta' >0$, in which $\omega_n = \langle \omega, e^{in y} \rangle_{L^2(\mathbb{T})}$. 
For {\bf (H.1)}, we note that if $\omega$ is an eigenfunction for $\mathcal{L}$ with an eigenvalue $\lambda$, that is $\mathcal{L} \omega = \lambda \omega$, then 
$W= \begin{pmatrix} \omega \\ \partial_y \omega \\ \partial_z \omega \end{pmatrix} $ is an eigenfunction for $L$ with the same eigenvalue $\lambda$. In addition, the growing mode must be of the form given by Lemma~\ref{lem-unstable}. %$e^{i\alpha (y-cs) }\hat \omega(z)$ with $\hat \omega = \partial_z^2 \hat \varphi$, where $\hat \varphi$ solves the Rayleigh problem \eqref{Rayleigh}. 
The regularity of $W$ follows from that of the given shear flow $U(z)$, which is real analytic; see \eqref{U-analytic}. %Moreover, as a consequence of the real-analyticity, $W$ is not identically $0$ in an open set. 
%The assumption {\bf (H.2)} follows. 
%By Lemma \ref{lem-unstable}, $\lambda_0 = -i n c_0$ is an unstable eigenvalue of $L$, for all $n \in \mathbb{Z}$ with $n \Im c_0>0$. 
The definition of $\gamma_0$ in~\eqref{def-g0}, Lemma \ref{lem-unstable}, and Proposition~\ref{prop-semigroup-cL-e} finally prove that {\bf (H.2)} holds.

Assumption {\bf (H.3)} follows directly from the definition of $L$, whereas {\bf (H.4)} and {\bf (H.5)} are clear, thanks to the structure of the quadratic nonlinearity $Q(W,W)$ and the fact that there are no $R_1, R_2$ generated from the system \eqref{pert-vortM}.

\section{Ill-posedness of the kinetic incompressible Euler equations}
\label{s-ill}

In this section, we give the proof of Theorem \ref{t-ill}, establishing ill-posedness for the kinetic incompressible Euler equations~\eqref{Vlasov}-\eqref{Euler}, which we recall below for convenience:%study the quasineutral limit for Vlasov-Poisson system: 
\begin{equation}\label{nVM1} 
\partial_t g + v \cdot \nabla_x g - \nabla_x \Phi \cdot \nabla_v g  =0
\end{equation}
with \begin{equation} \label{ellip12}
\begin{aligned}
  -\Delta_x \Phi = \nabla_x \cdot\Big (\nabla_x \cdot \Big( \int v \otimes v g \, dv\Big)\Big), \qquad \int_{\mathbb{T}^3}\Phi \; dx  =0.
 \end{aligned}\end{equation}
We shall prove in this section how Theorem~\ref{t-ill} follows from the abstract Theorem~\ref{thm-abstract}.
 We work with the analytic function space $X_{\delta, \delta'}$, equipped with the following norm: 
\begin{equation}\label{def-norm}
 \| f\|_{\delta, \delta'}: =\sum_{n\in \mathbb{Z}^3} \sum_{|\alpha|\ge 0} \| \langle v\rangle^m  \partial_v^\alpha f_n \|_{L^2( \mathbb{R}^3\times \mathbb{R}^3)} \frac{|\delta'|^{|\alpha|} }{ |\alpha|!} e^{\delta |n| } ,\end{equation}
for any $\delta, \delta' >0$, in which $f_n = \langle f, e^{in \cdot y} \rangle_{L^2(\mathbb{T}^3)}$. Here, $m$ is a fixed number, $m\ge 4$. We also introduce the $v$-analytic function space $X_{\delta'}$, equipped with the norm
\begin{equation}
 \| f\|_{\delta'}: = \sum_{|\alpha|\ge 0} \| \langle v\rangle^m \partial_v^\alpha f \|_{L^2(\mathbb{R}^3)} \frac{|\delta'|^{|\alpha|}}{|\alpha| !} 
 \end{equation}
for any $\delta'>0$.  
 %Here, $\tilde j = \int v g\; dv$, $\tilde m = \int v \otimes v g \, dv$, and $ \nabla \cdot(\nabla \cdot \tilde m) = \sum_{j,\ell} \partial_{x_\ell} \partial_{x_j} \tilde m_{j\ell}$. 
 We study ill-posedness near radial homogeneous equilibria of the form $g = \mu(v) \equiv \mu(|v|^2)$ and $\Phi= 0$, for real analytic functions $\mu$ satisfying 
 $ \int \mu(v) \; dv = 1$ and $\| \mu \|_{\delta'} <+\infty$ for some $\delta'>0$.
%In what follows, we shall consider the radial homogenous equilibria of the form: $\mu = \mu(|v|)$. In addition, we 

%We shall show that the system \eqref{nVM1} is ill-posed for initial data near \emph{unstable} equilibria $\mu(|v|)$. To do so,
We write the perturbed solution in the fast variables as follows: 
$$
g (t,x,v)= \mu(v) + f(s,y,v), \qquad \Phi (t,x)= \varphi (s,y), \qquad s =t /\varepsilon, \qquad  y = x /\varepsilon.
$$
The new pair $(f,\varphi)$ then solves 
 \begin{equation}\label{nVM4} 
\left \{ 
\begin{aligned}
\partial_s f &+ v \cdot \nabla_y f - \nabla_y \varphi \cdot \nabla_v \mu - \nabla_y \varphi \cdot \nabla_v f 
=0,
\\
  -\Delta_y \varphi &=\nabla_y \cdot\Big (\nabla_y \cdot \Big( \int v \otimes v f \, dv\Big)\Big), \qquad \int_{\mathbb{T}^3}\varphi \; dy  =0.
\end{aligned}\right.
\end{equation}
We shall show that the problem \eqref{nVM4} is ill-posed due to the unbounded unstable spectrum of the linearized operator
 $$
 \mathcal{L} f :=  - v \cdot \nabla_y f + \nabla_y \varphi \cdot \nabla_v \mu, \qquad   -\Delta_y \varphi = \nabla_y \cdot\Big (\nabla_y \cdot \Big( \int v \otimes v f \, dv\Big)\Big),\qquad  \int_{\mathbb{T}^3}\varphi \; dy  =0,
 $$
which we shall study in details in the next section. Next, since the nonlinearity $\nabla_y \varphi \cdot \nabla_v f$ is quadratic with respect to the partial derivatives of $f$, we are led to write the problem \eqref{nVM4} in the matrix form 
for the vector: 
\begin{equation}F: = \begin{pmatrix} f \\ \nabla_y f \\ \nabla_v f \end{pmatrix}.  \end{equation} 
We write $F = (F_1, F_2, F_3)\in \mathbb{R}\times \mathbb{R}^3 \times  \mathbb{R}^3$. We introduce the matrix operator
\begin{equation}\label{def-cL}
L: = \begin{pmatrix} \mathcal{L} & 0 & 0 \\ 0 &   \mathcal{L}_3 & 0 \\  0 & \mathcal{M} & \mathcal{T} \end{pmatrix},
\end{equation}
in which $$\mathcal{L}_3  = \begin{pmatrix} \mathcal{L} & 0 & 0 \\ 0&\mathcal{L}&0 \\ 0&0&\mathcal{L}\end{pmatrix}, \qquad \mathcal{T} = - v \cdot \nabla_y,\qquad 
\mathcal{M} G :=  - G + \nabla_v ( \varphi(G)\cdot \nabla_v \mu).
$$
Here, the vector $\varphi(G) = (\varphi(G_1), \varphi(G_2), \varphi(G_3))$ is understood as the unique solution to the elliptic problem
$$ -\Delta_y \varphi(G_k)  = \nabla_y \cdot \left(\nabla_y \cdot \left(\int v \otimes v G_k(v) \, dv\right)\right),$$ 
with zero average over $\mathbb{T}^3$, for each vector field $G(v)$.

It follows that $f$ solves \eqref{nVM4} if and only if the vector field $F = (F_1, F_2, F_3)$ solves 
\begin{equation}\label{eqs-F}
\partial_s F -  L F = Q(F,F) ,
\end{equation}
in which direct calculations show 
$$
Q(F,F) =  \begin{pmatrix} \varphi(F_2) \cdot F_3 \\  \nabla_y (\varphi(F_2) \cdot F_3) \\ \nabla_v (\varphi(F_2) \cdot F_3)  \end{pmatrix} .
$$
We shall show the ill-posedness of \eqref{eqs-F} by directly checking the assumptions {\bf (H.1)}--{\bf (H.5)} made in our abstract ill-posedness framework. 

 \subsection{Unbounded unstable spectrum  of the linearized operator}%{Unbounded unstable spectrum}
In this section, we study the linearized problem:
\begin{equation}\label{nVP-lin} 
\partial_s f + v \cdot \nabla_y f - \nabla_y \varphi \cdot \nabla_v \mu  =0, \qquad  -\Delta_y \varphi = \nabla_y \cdot \Big (\nabla_y \cdot \Big( \int v \otimes v f \, dv\Big)\Big).
\end{equation}
We search for a possible growing mode of the form:  
\begin{equation}\label{def-grmode-nVP}
(f,\varphi) = (e^{in\cdot (y-\omega t)}\hat f(v),  e^{in\cdot (y-\omega t)}\hat \varphi)\end{equation} for some complex constant $\hat \varphi$
 and complex function $\hat f(v)$, with $\Im (n \cdot \omega) >0$,  for some complex vector $\omega$. Plugging the above ansatz into the Vlasov equation in \eqref{nVP-lin}, we get 
 $$ in \cdot ( v - \omega ) \hat f  -in  \hat \varphi \cdot  \nabla_v \mu= 0 $$  
which gives \begin{equation}\label{clVM-def-lim} \hat f  =   \frac{ \nabla_v \mu \cdot n}{n\cdot (v-\omega)}  \hat \varphi , \qquad  \hat \varphi = \frac{-1}{|n|^2} \sum_{j,\ell} n_j n_\ell \int v_j v_\ell \hat f (v) \; dv.\end{equation}
This yields the existence of a growing mode if and only if there is a pair $(n,\omega)$, with $\Im (n\cdot \omega) >0$, so that the following dispersion relation holds:  
\begin{equation}\label{Penrose2}
\frac{1}{|n|^2} \sum_{j,\ell} n_j n_\ell  \int v_j v_\ell   \frac{ \nabla_v \mu \cdot n}{n\cdot (v-\omega)}   \; dv = -1 .
\end{equation}
%For the sake of simplicity, we shall restrict to the spatial frequency in the form $n = (n,0,0)$ and so $\omega = (\omega, 0,0)$. The dispersion relation simply becomes 
%\begin{equation}\label{Penrose2}
%  \int_{\mathbb{R}^3} \frac{v_1^2 \mu_{v_1}}{v_1 - \omega} \; dv = -1,\end{equation}
%for some $\omega$ with $\Im \omega\not =0$. 
We shall call this property  the Penrose instability condition.

We summarize this statement in the following lemma:

\begin{lemma}\label{lem-unmode2} The linearized operator $\mathcal{L}$ possesses a growing mode in the form \eqref{def-grmode-nVP} if and only if the Penrose instability condition \eqref{Penrose2} holds for some complex number $\Im \omega \not =0$. In case of instability, the unstable spectrum is unbounded, containing all the points $\lambda = - i n \cdot \omega$, with $n \in \mathbb{Z}^3$ so that $ \Im( n\cdot \omega)>0$ and with corresponding eigenfunctions given by~\eqref{def-grmode-nVP}-\eqref{clVM-def-lim}.
%$$ e^{i n y_1} \hat f(v), \quad \hat f(v)=   \frac{ \mu_{v_1} }{v_1-\omega}.$$
%whose total density $\hat \rho = \int \hat f \; dv = 1$.  
\end{lemma}

\subsection{Sharp semigroup bounds}
From now on, we consider a smooth radial equilibrium $\mu$ such that $\| \mu \|_{\delta'_1} <+\infty$ for some $\delta'_1>0$ and which gives unstable spectrum for \eqref{nVP-lin}. Typical examples are analytic radial double-bump equilibria with fast decay at infinity. Let $L$ be the matrix operator defined as in \eqref{def-cL}. We shall derive sharp bounds on the corresponding semigroup $e^{Ls}$ in the analytic function space $X_{\delta, \delta'}$.

Let us introduce, for all $n\in \mathbb{S}^2$,
$$ \mathcal{L}_{\hat n} f =  -i \hat n \cdot \Big(  v f  - \nabla_v \mu \varphi (f)\Big)  , \qquad   \varphi(f) := -\sum_{j,\ell} \hat n_j \hat n_\ell \int v_j v_\ell f (v) \; dv $$
We set $\gamma_0$ to be defined by 
\begin{equation}\label{def-g0-k} \gamma_0 : = \sup_{\hat n \in \mathbb{S}^2, \, \exists k \in \mathbb{N}^*, \, \sqrt{k} {\hat n} \in \mathbb{Z}^3} \Big \{ \Re \lambda_{\hat n} ~:~ \lambda_{\hat n} \in \sigma (\mathcal{L}_{\hat n})\Big\}.\end{equation}
%in which $ \mathcal{L}_{\hat n} f :=  -i \hat n \cdot (  v f  - \nabla_v \mu \varphi (f) )$, with $\varphi(f) = -\sum_{j,\ell} \hat n_j \hat n_\ell \int v_j v_\ell f (v) \; dv $.

Let us first quickly show that $\gamma_0$ exists and is positive. %We first ensure that the above set contains at least a positive element.  
Since $|\varphi (f)| \le C_0\| \langle v\rangle^4 f\|_{L^2}$ and $\mu$ decays sufficiently fast in $v$, it follows that $ i \hat n \cdot \nabla_v \mu \varphi (f) $ is a compact perturbation of the multiplication operator $-i \hat n \cdot v$. As a consequence, the unstable spectrum of $\mathcal{L}_{\hat n}$ consists precisely of  possible eigenvalues $\lambda$, solving the equation  $(\lambda - \mathcal{L}_{\hat n}) f  = 0$. It follows directly that $\lambda = -i \hat n \cdot \omega$ is an eigenvalue of $\mathcal{L}_{\hat n}$ if and only if 
the function $$D(\omega ; \hat n): =  1+ \sum_{j,\ell} \hat n_j \hat n_\ell  \int v_j v_\ell   \frac{ \nabla_v \mu \cdot \hat n}{\hat n\cdot (v-\omega)}   \; dv $$
has a zero $\omega \in \mathbb{C}^3$, for some $\hat n  = \frac{n}{|n|}$. We observe that $\sup_{\hat n \in \mathbb{S}^2} D(\omega; \hat n)\to 1$ as $|\omega|\to \infty$, and thus possible eigenvalues must lie in a bounded domain in the complex domain. Since we assume the existence of  unstable spectrum for \eqref{nVP-lin}, 
the above set is not empty, and $\gamma_0$ is well-defined and positive.
%there must be a maximal growing mode of $\mathcal{L}_\mathrm{\hat n}$ which  largest real part. This proves the existence and  positivity of $\gamma_0$. 

%The above maximum exists, since the Rayleigh problem has no solution as $|c|\to \infty$ and $D(c)$ is analytic in $c$, with $\Im c>0$. Let $c_0$ be the solution of $D(c_0) =0$ so that $\gamma_0 = \Im c_0$, and let $\hat \varphi_0$ be the corresponding eigenfunction. Since $\hat \varphi_0$ is a linear combination of $\hat \varphi_1, \hat \varphi_2$ defined as above, the regularity of $\hat \varphi_0$ follows from that of the given shear flow $U(z)$. 
%
\begin{proposition}\label{prop-semigroup-cL} Let $\gamma_0$ be defined as in \eqref{def-g0-k}, $\delta>0, $ and $\delta' \in (0,\gamma_0)$. The semigroup $e^{Ls}$, associated to $L$, is well-defined in $X_{\delta, \delta'}$, for $s$ and $\delta'$ small enough.
%Let  $\gamma > \gamma_0$ and $s>0$ such that $\delta - \gamma s> 0$. 
%For any function $h$ in $X_{\delta, \delta'}$, define $e^{Ls}h$ as the solution of
%
 %The matrix operator $L$ generates a semigroup $e^{Ls}$ for functions in $X_{\delta, \delta'}$. 
%Furthermore, if we assume that 
%\begin{equation}\label{def-g0} \gamma_0 : = \max_{\hat n \in \mathbb{S}^3} \Big \{ \Re \lambda_{\hat n} ~:~ \lambda_{\hat n} \in \sigma (\mathcal{L}_{\hat n})\Big\} \end{equation}
%is positive, in which $ \mathcal{L}_{\hat n} f :=  -i \hat n \cdot (  v f  - \nabla_v \mu \varphi (f) )$, with $\varphi(f) = -\sum_{j,\ell} \hat n_j \hat n_\ell \int v_j v_\ell f (v) \; dv $, 
More precisely, for any $\gamma > \gamma_0$, there is a positive constant $C_\gamma$ so that 
$$ \| e^{Ls} h \|_{\delta - \gamma s, \delta'} \le C_\gamma  \| h\|_{\delta, \delta'},$$
for all $h \in X_{\delta, \delta'}$,  and for all $\delta' \leq \min( \delta'/2, \gamma_0)$ and $s$ so that $\delta - \gamma s> 0$.  
\end{proposition}

We start the proof of Proposition \ref{prop-semigroup-cL} with the semigroup $e^{\mathcal{L}s}$. 
Here, we recall that  
$$\mathcal{L}f = 
- v \cdot \nabla_y f + \nabla_y \varphi \cdot \nabla_v \mu , \qquad  -\Delta_y \varphi = \nabla_y \cdot\Big (\nabla_y \cdot \Big( \int v \otimes v f \, dv\Big)\Big).$$
In Fourier variables (with respect to $y$), we study for all $n\in \mathbb Z^3$
$$ \mathcal{L}_n f =  -i n \cdot \Big(  v f  - \nabla_v \mu \varphi (f)\Big)  , \qquad   \varphi(f) := \frac{-1}{|n|^2} \sum_{j,\ell} n_j n_\ell \int v_j v_\ell f (v) \; dv $$
and solve the ODEs 
$$ (\partial_s - \mathcal{L}_n) f =0, \quad f(0,v) = f_0(v).$$ 
%
%We use the following analytic norm in $v$, with the analyticity radius $\delta'>0$: 
%$$ \| f\|_{\delta'}: = \sum_{|\alpha|\ge 0} \| \langle v\rangle^m  \partial_v^\alpha f\|_{L^2(\mathbb{R}^3)} \frac{|\delta'|^{|\alpha|}}{|\alpha| !} .$$
% Let $X_{\delta'}$ be the analytic space with the above finite norm: $\| \cdot \|_{\delta'}$. By definition, one can check that 
%$$ \| \nabla_v f \|_{\delta''} \le \frac{C (\delta')}{\delta' - \delta''} \| f\|_{\delta'}, \qquad \forall 0< \delta'' < \delta'.$$

\begin{lemma}\label{lem-semigroupLn} Let $\gamma_0$ be defined as in \eqref{def-g0-k}. For each $n\in \mathbb{Z}^3$, the operator $\mathcal{L}_n$ generates a continuous semigroup $e^{\mathcal{L}_n t}$  from $X_{\delta'}$ to itself. In addition, for any $\gamma>\gamma_0$, there is a positive constant $C_\gamma$ so that 
$$ \| e^{\mathcal{L}_n s} h \|_{\delta'} \le C_\gamma e^{\gamma|n|s} \| h\|_{\delta'},\qquad \forall~ s\ge 0,$$
for all $h \in X_{\delta'}$, for small $\delta'>0$. 
\end{lemma}

\begin{proof} 
%Let $n \in \mathbb{Z}^3$ and $m\ge 4$. 
We introduce the time scaling $s \mapsto |n| s$. It suffices to study the semigroup $e^{\mathcal{L}_{\hat n} s}$ for the scaled operator
$$ \mathcal{L}_{\hat n} f =  -i \hat n \cdot \Big(  v f  - \nabla_v \mu \varphi (f)\Big)  , \qquad   \varphi(f) := -\sum_{j,\ell} \hat n_j \hat n_\ell \int v_j v_\ell f (v) \; dv $$
for $\hat n = \frac{n}{|n|}$. Let $R$ be the rotation matrix so that $R \hat n = \hat n_1: = (1,0,0)$.  Since $\mu \equiv \mu(|v|^2)$, the operator $\mathcal{L}_{\hat n}$ is invariant under the change of variable: $\hat n \mapsto R\hat n$ and $v \mapsto R v$. Hence,  it suffices to derive estimates for $\mathcal{L}_1: = \mathcal{L}_{\hat n_1} = -iv_1 + \mu_{v_1} \varphi(\cdot)$. Since $|\varphi (f)| \le C_0\| \langle v\rangle^4 f\|_{L^2}$ and $\mu_{v_1}$ decays sufficiently fast in $v$, $ \mu_{v_1} \varphi(f)$ is a compact perturbation of the multiplication operator by $-iv_1$. Hence $\mathcal{L}_1$ generates a continuous semigroup $e^{\mathcal{L}_1s}$ with respect to the weighted norm $\| \langle v\rangle^m \cdot \|_{L^2(\mathbb{R}^3)}$.

 In addition, following the standard semigroup theory (\cite{Pazy} or \cite[Appendix A]{Zum}), we can write  
\begin{equation}
\label{eLn}
e^{\mathcal{L}_1s} h= P.V. \frac{1}{2\pi i} \int_{\gamma - i\infty}^{\gamma + i \infty} e^{\lambda s} (\lambda - \mathcal{L}_1)^{-1} h \; d\lambda \end{equation}
for any $\gamma > \gamma_0$. With $\mathcal{L}_1 = -iv_1 + \mu_{v_1} \varphi(\cdot)$, the resolvent equation $(\lambda - \mathcal{L}_1)f = h $ gives 
\begin{equation}\label{res-Lnh} f + \frac{ i \mu_{v_1}}{\lambda  + i v_1} \varphi(f) = \frac{h}{\lambda  + iv_1 } .\end{equation}
By a view of $\varphi(f)$, we can first solve $\varphi(f)$ in terms of the initial data $h$: 
$$  \varphi(f) = - \frac{1}{D(\lambda)}\int \frac{v_1^2 h}{\lambda  + i v_1} \; dv ,\qquad D(\lambda): =  1- i\int  \frac{ v_1^2 \mu_{v_1}}{\lambda  + i v_1}\; dv  .$$
We note that $D(\lambda) = 0$ if and only if $\lambda $ is an eigenvalue of $\mathcal{L}_1$. It follows that  
\begin{equation}\label{bd-phif}|\varphi(f) |\le \frac{C_\gamma}{1+|\Im \lambda| } \| \langle v\rangle^4 h \|_{L^2(\mathbb{R}^3)} \end{equation}
uniformly for all $\lambda = \gamma+ i\mathbb{R}$, with any fixed number $\gamma>\gamma_0$. Thus, from \eqref{eLn} and \eqref{res-Lnh}, we compute 
$$\begin{aligned}
e^{\mathcal{L}_1s} h
&= P.V. \frac{1}{2\pi i} \int_{\gamma - i\infty}^{\gamma + i \infty} e^{\lambda s} 
\Big[ -\frac{ i \mu_{v_1}}{\lambda  + i v_1} \varphi(f) + \frac{h}{\lambda  + iv_1 } \Big]  \; d\lambda
\end{aligned}$$
in which the second integral is equal to $ e^{ - i v_1 s} h$, while the first integral can be estimated directly using the estimate \eqref{bd-phif} on $\varphi(f)$. This yields at once 
\begin{equation}\label{L2-semigroup}\| \langle v\rangle^m e^{\mathcal{L}_1 s} h\|_{L^2} \le C_\gamma e^{\gamma s}   \| \langle v \rangle^m h \|_{L^2} ,\end{equation}
for any $\gamma>\gamma_0.$ 
%where we recall that $\gamma_0$ is defined as in \eqref{def-g0-k}. 
 
Next, for higher derivatives of $f = e^{\mathcal{L}_1 s}h$, we note that $ \partial_v^\alpha f$ solves 
$$ \partial_s  \partial_v^\alpha f + i v_1   \partial_v^\alpha f + i  \partial_v^\alpha \mu_{v_1} \varphi(f)  + i  [ \partial_v^\alpha, v_1] f =0.$$
Standard $L^2$ estimates for $ \partial_v^\alpha f$ yield, for all $\alpha = (\alpha_1, \alpha_2, \alpha_3)$ and $\alpha' =(\alpha_1-1, \alpha_2, \alpha_3)$,
\begin{equation}\label{Dv-k}
 \begin{aligned}
 \frac 12\frac{d}{ds} \| \langle v \rangle^m  \partial_v^\alpha f (s)\|_{L^2}^2 
 & \le   \Big [\|  \partial_v^\alpha\mu_{v_1} \varphi(f)\|_{L^2}  + \|   [ \partial_v^\alpha, v_1]  f\|_{L^2} \Big] \| \langle v \rangle^m  \partial_v^\alpha f(s)\|_{L^2}
 \\
 & \le   \Big [C_0\| \langle v\rangle ^m f \|_{L^2} \|   \partial_v^\alpha \mu_{v_1} \|_{L^2} + |\alpha_1| \|    \partial_v^{\alpha'} f\|_{L^2} \Big] \| \langle v \rangle^m   \partial_v^\alpha f(s)\|_{L^2}
,
 \end{aligned}\end{equation}
 upon using the fact that the term $iv_1 f$ does not yield any contribution when taking the real part of the $L^2$ energy identities. By a view of the definition of the analytic norm, the above estimates give 
$$ \begin{aligned}
\frac{d}{ds} \| f(s)\|_{\delta'} 
&= \sum_{|\alpha|\ge 0} \frac{d}{ds} \| \langle v\rangle^m   \partial_v^\alpha f(s)\|_{L^2(\mathbb{R}^3)} \frac{|\delta'|^{|\alpha|}}{|\alpha| !}
\\
&\le  C_0  \| \langle v \rangle^m f \|_{L^2} +  \sum_{|\alpha|\ge 1}  \frac{|\delta'|^{|\alpha|}}{{|\alpha|} !}  \Big [ C_0 \| \langle v\rangle ^m f \|_{L^2} \|   \partial_v^\alpha \mu_{v_1}\|_{L^2} + |\alpha_1| \|   \partial_v^{\alpha'} f\|_{L^2} \Big]
\\
&\le  C_0  (1+ \| \nabla_v \mu\|_{\delta'}) \| \langle v \rangle^m f \|_{L^2} +  \delta'  \sum_{|\alpha'|\ge 1}  \frac{|\delta'|^{|\alpha'|}}{|\alpha'| !}  \|   \partial_v^{\alpha'}f\|_{L^2} 
\\
&\le C_0  (1+ \| \nabla_v \mu\|_{\delta'}) \| \langle v \rangle^m f \|_{L^2} + \delta' \|  f \|_{\delta'} ,
\end{aligned}$$
which entails
\begin{equation}\label{Af-bound0} \| f(s)\|_{\delta'} \le  e^{\delta' s} \| f(0)\|_{\delta'} + C_0  (1+ \| \nabla_v \mu\|_{\delta'}) \int_0^s  e^{\delta' (s-\tau)} \| \langle v \rangle^m f(\tau) \|_{L^2} \; d\tau ,\end{equation}
for any $\delta'>0$. Now, thanks to the $L^2$ bound \eqref{L2-semigroup} on the semigroup and the assumption that $\delta'\le \gamma_0$, we get 
\begin{equation}\label{semigroup-bound0} \| f(s)\|_{\delta'} \le  \tilde{C}_\gamma  (1+ \| \nabla_v \mu\|_{\delta'}) e^{\gamma s} \| f(0)\|_{\delta'} 
.\end{equation} 
%Using the rescaling in time $s \mapsto |n| s$, 
The claimed bound in the lemma is therefore proved. 

\end{proof}

We can finally end the proof of Proposition \ref{prop-semigroup-cL}.

\begin{proof}[Proof of Proposition \ref{prop-semigroup-cL}] We let $ f = e^{\mathcal{L}s} h$. Lemma \ref{lem-semigroupLn} yields 
$$ \| f_n  \|_{\delta'} \le C_\gamma e^{\gamma  |n|s} \| h_n\|_{\delta'},$$
for all $\gamma >\gamma_0$, and for small enough  $\delta' >0$. Hence, by definition of the norms, for any $s$ so that $\delta - \gamma s >0$, we get 
$$ \| f(s)\|_{\delta - \gamma s, \delta'} \le  C_\gamma \sum_{n\in \mathbb{Z}^3} e^{\gamma |n|s} \| h_n\|_{\delta'}  e^{(\delta - \gamma s) |n| } \le C_\gamma \| h\|_{\delta , \delta'}.$$
This proves the claimed bound for $e^{\mathcal{L}s}$. As for $F = e^{L s}H$, by the structure of the matrix operator $L$ (see \eqref{def-cL}), it is clear that the above estimate holds for $F_1, F_2$. As for $F_3$, we write 
$$( \partial_s + v \cdot \nabla_y )F_3  +  F_2  -  \nabla_v ( \varphi(F_2)\cdot \nabla_v \mu)  = 0.$$ 
Similarly to the estimate obtained in \eqref{Af-bound0} via energy estimates, for $\delta' \le \gamma_0$, in the Fourier variable $n \in \mathbb Z^3$, we immediately get 
$$
\begin{aligned}
 \| F_{3,n}(s)\|_{\delta'}  
 &\le  e^{\delta' |n|s} \| H_{3,n}(0)\|_{\delta'} + C_0  (1+ \| \nabla_v \mu\|_{\delta'}) \int_0^s  e^{\delta' |n|(s-\tau)} \| F_{2,n}(\tau)\|_{\delta'}  \; d\tau
 \\
 &\le  e^{\delta' |n|s} \| H_{3,n}(0)\|_{\delta'} + C_\gamma  (1+ \| \nabla_v \mu\|_{\delta'}) \int_0^s  e^{\delta' |n|(s-\tau)} e^{\gamma |n|\tau} \| H_{2,n}\|_{\delta'}  \; d\tau
 \\
 &\le  e^{\delta' |n|s} \| H_{3,n}(0)\|_{\delta'} + C_\gamma  (1+ \| \nabla_v \mu\|_{\delta'})  e^{\gamma |n|s} \| H_{2,n}\|_{\delta'} .
 \end{aligned}
$$
Hence, as above, summing  the norms for all $n \in \mathbb{Z}^3$, we obtain the claimed bound for $F_3$, and hence for $e^{Ls}$. This completes the proof of the proposition.
\end{proof}

\subsection{Conclusion}
We are now ready to check the assumptions {\bf (H.1)}--{\bf (H.5)} made in the abstract framework. 
% First, we recall the analyticity norm: 
%\begin{equation*}
% \| f\|_{\delta, \delta'}: =\sum_{n\in \mathbb{Z}^3} \sum_{|\alpha|\ge 0} \| \langle v\rangle^m   \partial_v^\alpha f_n \|_{L^2(\mathbb{R}^3\times \mathbb{R}^3)} \frac{|\delta'|^{|\alpha|} }{|\alpha| !} e^{\delta |n| } , \end{equation*}
%for any $\delta, \delta' >0$, in which $f_n = \langle f, e^{in \cdot y} \rangle_{L^2(\mathbb{T}^3)}$. 
%The assumption {\bf (H.1)} follows by the norm definition. 

For {\bf (H.1)}, we note that if $g$ is an eigenfunction for $\mathcal{L}$ with an eigenvalue $\lambda$, that is $\mathcal{L} g = \lambda g$, then 
$G= \begin{pmatrix} g \\ \nabla_y g \\ \nabla_v g \end{pmatrix} $ is an eigenfunction for $L$ with the same eigenvalue $\lambda$. Thus, by construction of the growing mode of $\mathcal{L}$ in Lemma \ref{lem-unmode2}, the very definition of $\gamma_0$ in~\eqref{def-g0-k}  and Proposition \ref{prop-semigroup-cL}, {\bf (H.2)} holds. Assumption {\bf (H.3)} follows directly from the definition of $L$, whereas {\bf (H.4)} and {\bf (H.5)} are clear, thanks to the structure of the quadratic nonlinearity $Q(F,F)$ and the fact that there are no $R_1, R_2$ generated from the system \eqref{eqs-F}.

%TODO: uniqueness argument to show that solving this bigger system really gives a solution to the original one.

\bigskip

\noindent {\bf Acknowledgements.} We are grateful to Pennsylvania State University and \'Ecole polytechnique for hospitality during the preparation of this work. TN was partly supported by the NSF under grant DMS-1405728 and by a two-month visiting position at \'Ecole polytechnique during Spring 2015.

\bibliographystyle{plain}
\bibliography{eMHD}

\end{document}